\documentclass[reqno, 11pt]{amsart}
\usepackage[top=4cm, bottom=3cm, left=3.2cm, right=3.2cm]{geometry}
\usepackage{enumerate}
\usepackage{amssymb,amsmath}
\usepackage{xcolor}
\usepackage{lineno}
\usepackage{hyperref}
\usepackage{enumitem}

\newtheorem{thm}{Theorem}

\newtheorem{lem}{Lemma}

\newtheorem{assumption}{Assumption}
\theoremstyle{definition}
\newtheorem{defn}{Definition}

\allowdisplaybreaks

\numberwithin{equation}{section} \numberwithin{lem}{section}
\numberwithin{thm}{section} \numberwithin{prop}{section}
\numberwithin{cor}{section} \numberwithin{rem}{section}
\numberwithin{defn}{section}

	\newtheorem{remark}{Remark}[section]

	\catcode`@=11 \@addtoreset{equation}{section} \catcode`@=12
\title[Stochastic parabolic-parabolic Keller-Segel system]{Global Existence and Pathwise Uniqueness for a Stochastic Parabolic-Parabolic Keller-Segel System}

\author{ Jinhuan Wang ${^1}$, Qian Li ${^1}$, Hui Huang ${^{2, *}}$}
\thanks{The work of J. Wang is partially supported by National Natural Science Foundation of China Grants No. 12171218, Liaoning Provincial Natural Science Foundation Program (2024-MS-002). H.H. is partially supported by the Start-up grant from Hunan University.}
\thanks{* Corresponding author: Hui Huang }
\begin{document}
\maketitle
\begin{center}
{\footnotesize
1-School of Mathematics and Statistics, Liaoning University, Shenyang, 110036, P. R. China \\
Email: wangjh@lnu.edu.cn; liqian990720@163.com  \\
 \smallskip
2-School of Mathematics, Hunan University, Changsha, 410000, P. R. China
\\
Email: huihuang1@hnu.edu.cn
}
	\end{center}
\maketitle
 \date{} 

\begin{abstract}
In this paper, we study a stochastic parabolic-parabolic Keller-Segel system driven by nonlocal, nonlinear multiplicative noise in a two-dimensional bounded domain. Under suitable assumptions, we establish global existence and pathwise uniqueness of a strong solution for arbitrary initial data, without imposing any smallness conditions. This sharply contrasts with the deterministic two-dimensional Keller-Segel system, which typically requires smallness assumptions on initial mass.

The main analytical challenges stem from the fully parabolic coupling, combined with a lack of coercivity and global Lipschitz continuity in both the chemotactic drift and noise terms. To overcome this, we introduce a tailored truncated system to establish local existence via Banach's fixed point theorem. Using this local existence and pathwise uniqueness, we construct a maximal local strong solution. Finally, by introducing a specialized Lyapunov functional, we derive uniform estimates to extend this solution globally.

\end{abstract}
\maketitle

\maketitle
 \date{}



\section{Introduction}
 Chemotaxis is the directed movement of cells or organisms in response to chemical signals. A classical mathematical model for  this phenomenon is the Keller-Segel system,  originally introduced by Keller and Segel \cite{keller1970initiation, keller1971model} to describe the aggregation of biological species.  More precisely, this model describes organisms that, in addition to random diffusion, tend to move toward regions of higher concentration of nutrients or chemical substances secreted by the organisms themselves. 
A substantial amount of work has been devoted to the mathematical analysis of chemotaxis models, and considerable progress has been made in the study of global existence, finite-time blow-up, and long-time behavior in both bounded and unbounded domains. Among these models, the Keller-Segel system plays a central role and has been extensively studied over the past decades. As a representative example, the classical deterministic Keller-Segel system takes the form
\begin{eqnarray}\label{classical_K_S}
\begin{cases}	
	u_t  = \Delta u-  \nabla  \cdot (u\, \nabla v),&x\in\mathcal{O},t> 0,
    \\v_t=\Delta v-v+u,&x\in\mathcal{O},t> 0,\\\frac{\partial u}{\partial \nu} =\frac{\partial v}{\partial \nu} = 0,&x\in\partial\mathcal{O},t> 0,
    \\u (x,0) = u_{0}(x),\quad v (x,0) = v_{0}(x),&x\in\mathcal{O},
 \end{cases} 
\end{eqnarray}
where $\mathcal{O}\subset\mathbb{R}^d$ is a bounded domain with smooth boundary. In dimension $d=1$, Osaki and Yagi \cite{osaki2001finite} proved that all solutions of (\ref{classical_K_S}) are global in time and bounded. In dimension $d=2$,  the global existence of solutions of (\ref{classical_K_S}) is determined by the initial mass of $\int _\mathcal{O} u_0(x)dx$.  If $\int _\mathcal{O} u_0(x)dx<{4\pi}$, Nagai et al. \cite{nagai1997application} showed that the classical solution is global and bounded, whereas if $\int _\mathcal{O} u_0(x)dx>{4\pi}$ satisfying $\int _\mathcal{O} u_0(x)dx\notin \{4k\pi|k\in\mathbb{N}\}$, Horstmann and Wang \cite{horstmann2001blow} derived that the corresponding solution  blows up either in finite or infinite time. In  higher dimensions $d\ge 3$, Winkler \cite{winkler2010aggregation} showed that there exists $\varepsilon_0>0$, if  the  initial data satisfy $\|u_0\|_{L^p(\mathcal{O})}<\varepsilon_0$ with $p>\frac{d}{2}$ and  $\|\nabla v_0\|_{L^q(\mathcal{O})}<\varepsilon_0$ with $q>d$, the solution is global in time and bounded, and converges to the constant steady state determined by the total mass. Cao \cite{cao2015global} extended this result to the corresponding critical case, that is, $p=\frac{d}{2}$ and $q=d$. In addition,  for the parabolic-parabolic Keller–Segel system on the whole space $\mathbb{R}^2$, a similar mass threshold phenomenon is expected to determine whether solutions exist globally or blow up in finite time, with the conjectured critical mass being $8\pi $, see \cite{herrero1997blow,calvez2008parabolic,nagai2011brezis,mizoguchi2013global,corrias2006critical}. When the chemical substance diffuses sufficiently fast, the fully parabolic Keller–Segel system (\ref{classical_K_S}) is often approximated by the parabolic–elliptic chemotaxis model, where the chemical concentration satisfies an elliptic equation, for instance $-\Delta v=u- v$ and $-\Delta v=u$. The qualitative behavior of solutions, including global existence and finite-time blow-up, depends crucially on the properties of the initial data. We refer to \cite{blanchet2006two,nagai1997application,herrero1996chemotactic,blanchet2009critical,sugiyama2010blow} and the references therein for a more comprehensive account of related results.

 Keller-Segel type models for chemotaxis can arise naturally as mean-field limits of stochastic interacting particle systems in the large-population regime. Rigorous derivations of such limits can be found in \cite{stevens2000derivation,stevens2000stochastic}. However, in the deterministic Keller-Segel framework, the underlying particle system is typically driven only by mutually independent idiosyncratic noises, whose fluctuations vanish in the mean-field limit. Consequently, the corresponding macroscopic dynamics are deterministic; see \cite{jabin2017mean} for related results on propagation of chaos and mean-field limits.
 In more realistic settings, the particles may also be subject to common environmental fluctuations, such as changes in temperature, light, or sound. Since these effects persist in the mean-field limit, they are naturally reflected, at the macroscopic level, in stochastic Keller-Segel type models with noise. 
 For rigorous derivations in this direction, we refer to \cite{chen2025hegselmann, huang2021microscopic,nikolaev2025quantitative}. Motivated by this particle interpretation, stochastic Keller-Segel systems have attracted considerable attention in recent years. Existing studies have mainly focused on the well-posedness and qualitative behavior of solutions under various forms of stochastic forcing, corresponding to different choices of the noise coefficient \(\sigma(u)\). A representative model is given by  
\begin{equation}\label{classical_S_K}
\begin{cases}
du = \big(\Delta u-\nabla\cdot(u\nabla v)\big)\,dt+\sigma(u)\,dW(t),
& x\in\mathcal{O},\ t>0,\\
-\Delta v+v = u,
& x\in\mathcal{O},\ t>0.
\end{cases}
\end{equation}

A number of works have investigated systems of the form \eqref{classical_S_K} from different perspectives, with the spatial domain \(\mathcal{O}\) taken to be either the whole space \(\mathbb{R}^d\) or a bounded domain. For instance, 
 Huang and Qiu \cite{huang2021microscopic} studied the  system \eqref{classical_S_K} with divergence form noise in $\mathbb{R}^d\,(d=2,3)$  and proved the existence of a unique global weak solution provided that the norm of the initial data is sufficiently small in ${L^4(\mathbb{R}^d)}$, and further obtained strong solutions when the initial data possesses higher regularity. In $\mathbb{R}^2$, Misiats et al. \cite{misiats2022global}  showed that divergence form noise leads to global weak solutions for small initial mass, whereas non-divergence form noise may cause finite time blow-up with positive probability for any nonzero initial data. 
 In $\mathbb R^d$ ($d\ge2$), Tang and Wang \cite{tang2024strong} established local existence and uniqueness of strong solutions for a nonlocal aggregation-diffusion model   with multiplicative noise, and further showed that non-autonomous linear noise yields global existence and decay, whereas sufficiently strong nonlinear noise can prevent blow-up. In addition, in a bounded domain $\mathcal{O}\subset \mathbb{R}^2$,  Chen et al. \cite{chen2025well} established global well-posedness of  system \eqref{classical_S_K} with either a logistic source term or nonlinear noise, and proved the existence of a unique global nonnegative mild solution for any nonnegative initial data. 
 
 Compared with the stochastic parabolic-elliptic case,  stochastic parabolic-parabolic models have been much less investigated. To the best of our knowledge, one relevant contribution is due to Hausenblas et al. \cite{hausenblas2022one}, who constructed a stochastic version of the Schauder–Tychonoff fixed point theorem  and used it to study  a one-dimensional stochastic parabolic-parabolic Patlak-Keller-Segel system with Stratonovich noise acting in both the cell density equation and the chemical concentration equation. They established the existence of a martingale solution on any finite time interval.

Based on the above studies for deterministic and stochastic Keller-Segel systems,  we observe that  the stochastic parabolic-parabolic Keller-Segel model in which the random perturbation acts only on the cell density equation  has not been systematically investigated so far. Therefore, in this paper, we study the following stochastic parabolic-parabolic Keller–Segel system with a nonlocal nonlinear multiplicative noise:
\begin{equation}\label{equ1}
\begin{cases}
du = \big(\Delta u-\nabla\cdot(u\nabla v)\big)\,dt+\sigma(u)\,dW(t),
& x\in\mathcal{O},\ t>0,\\
dv = \big(\Delta v-v+u\big)\,dt,
& x\in\mathcal{O},\ t>0,\\
\dfrac{\partial u}{\partial \nu}
=
\dfrac{\partial v}{\partial \nu}
=0,
& x\in\partial\mathcal{O},\ t>0,\\
u(x,0)=u_0(x),\quad v(x,0)=v_0(x),
& x\in\mathcal{O},
\end{cases}
\end{equation}
where $\mathcal{O}\subset\mathbb{R}^2$ is a bounded domain with a  smooth boundary  and $\frac{\partial }{\partial \nu}$ denotes the derivative with respect to the outward unit normal vector $\nu$ of  $\partial\mathcal{O}$. 
 Here $u$ represents the bacterial density, $v$ represents the chemical substance concentration.  Furthermore, $\sigma(u)dW(t)$ represents the driving external random force, where $W(t)$ is cylindrical Wiener process.
 
Our study is further inspired by the nonlocal nonlinear noise structures introduced in the stochastic parabolic-elliptic Keller-Segel systems by     Tang and Wang \cite{tang2024strong} and  Chen et al. \cite{chen2025well}. More precisely, Tang and Wang considered a noise coefficient involving  a nonlocal factor of the form $\alpha(1+\|u\|_{H^r})^\rho u$  and Chen et al. studied multiplicative noise of the form $b_k \|u\|_{L^p}^ru$. Here, the constant $\alpha $ and $b_k$  specify the strength of the noise.  These works  indicate  that appropriately chosen nonlocal nonlinear noise may prevent finite-time blow-up and lead to global existence, in sharp contrast to the corresponding deterministic parabolic–elliptic model, where blow-up may occur for suitable initial data. 

A key feature of this type of noise lies in its nonlocal multiplicative structure. The intensity of the stochastic perturbation depends not only  on the local density $u(x)$, but also on the global $L^p$-size of the population. From the modeling point of view, this reflects the fact that environmental perturbations may be modulated by the overall population level or concentration.
From the analytical point of view,  such noise helps counteract the aggregation effect produced by the chemotactic drift, which is one of the key ingredients in the proof of global existence. 
 
Motivated by these observations, we investigate whether a similar nonlocal nonlinear noise can also prevent blow-up in the parabolic-parabolic setting. To this end, the present work is concerned with global existence and pathwise uniqueness for a stochastic parabolic-parabolic Keller-Segel system (\ref{equ1}) with a nonlocal nonlinear multiplicative noise coefficient of the form
$b_k\|u\|_{L^p(\mathcal O)}^r u$ in a two-dimensional bounded domain. It is worth emphasizing that the range of admissible parameters in the present work is different from that in \cite{chen2025well}.  The precise assumptions are stated in Assumption \ref{ass}. 

 In contrast to the deterministic parabolic-parabolic Keller-Segel system in a bounded  two-dimensional domain, for which the  global existence and blow-up  depend crucially on the critical initial mass $4\pi$, in the present work,  we establish the global existence and pathwise uniqueness of a solution of system \eqref{equ1} for arbitrary initial data in the present stochastic setting.  The main analytical difficulties arise from two aspects.  On the one hand, neither the chemotactic drift nor the nonlocal noise coefficient  satisfies the standard coercivity  or global Lipschitz conditions, so that standard existence theories for SPDEs are not directly applicable. On the other hand, the equation for the chemical concentration remains parabolic rather than elliptic, which makes the analysis require joint control of two fully coupled time-dependent unknowns.
 To overcome the above difficulties, we first introduce a truncated system \eqref{equ-u} and establish local existence by Banach's fixed point theorem. Pathwise uniqueness then allows us to construct a maximal local strong solution to the original system \eqref{equ1}. To prove global existence, we further introduce a suitable Lyapunov functional and derive uniform estimates, which eventually enable us to extend the maximal local solution globally in time.

We now briefly outline the main proof ideas, the novelties of our approach, and the organization of the paper.

In Section 2, we introduce the probabilistic framework, including some basic properties of the cylindrical Wiener processes, and state the assumptions on the noise coefficient $\sigma$. We also derive several auxiliary estimates that will be used throughout the paper and present the main global existence result.

In Section 3, we construct a global pathwise unique strong solution for arbitrary initial data. The proof is carried out in three steps.

\textbf{Step 1.}
We first establish the existence of local strong solutions to system \eqref{equ1}. Since neither the chemotactic drift nor the noise term satisfies the standard coercivity condition, and  since the regularity of $v$ has to be derived from its own time-dependent parabolic equation,  we consider the following
truncated system of (\ref{equ1}) under homogeneous Neumann boundary conditions
\begin{equation}\label{equ-u}
\begin{cases}
du_N=\Delta u_Ndt-\theta_N(\|u_N\|_{X_t}^2)\nabla\cdot\big(u_N\nabla v_N\big)dt+\theta_N(\|u_N\|_{X_t}^2)\sigma(u_N)dW(t),\\
d v_N=(\Delta v_N-v_N+u_N)dt,\\
u_N(0)=u_0, v_N(0)=v_0.
\end{cases}
\end{equation}
Here, for any $N>0$, we define $\theta_N(r)=\theta(r/N)$, where  $\theta:[0,\infty)\to[0,1]$ is a $C^2$ function such that $\theta(r)=1$ for $r\in [0,1]$, $\theta(r)=0$ for $r\ge2$ and $\sup_{r\in[0,\infty)}|\theta'(r)|\leq C_\theta<\infty$.
A distinctive feature of our truncation procedure is that it is performed with respect to the combined energy and gradient dissipation quantity. For every $t\ge 0$, we define 
\[
\|\rho \|_{X_t}^2:=\sup_{s\in[0,t]}\|\rho (s)\|_{L^2(\mathcal {O})}^2+\int_0^t\|\nabla \rho(s)\|_{L^2(\mathcal {O})}^2\,ds.
\]
For each fixed $N$, by the definition of $\theta_N(\cdot)$, if $\theta_N(\|\rho\|_{X_t}^2)\ne0$,  one has $\|\rho \|_{X_t}^2< 2N,$
which is essential for handling the fully parabolic coupling of system. Moreover, even after truncation, the system still does not satisfy the Lipschitz conditions required by standard existence theorems for SPDEs \cite [Theorem 4.2.4] {liu2015stochastic}. To overcome this difficulty, we freeze both the chemotactic term and the noise coefficient so that a contraction mapping can be established. As a result, for each $N\in \mathbb{N}$, the truncated system \eqref{equ-u} admits a global strong solution $u_N$.
Consequently, the original system \eqref{equ1} admits a local strong solution $u_N(\cdot\wedge\tau_N)$, where the stopping time $\tau_N$ is induced by the truncation.
  
  \textbf{Step 2.}
  On a fixed stochastic basis, we establish pathwise uniqueness of local strong solutions. This uniqueness property is then used to construct a maximal local solution to the original system \eqref{equ1}. The details are given in Subsection 3.2.

  \textbf{Step 3.}
   We show that the maximal local solution is in fact global in Subsection 3.3. The crucial point is to construct a suitable Lyapunov functional $\Phi(x)=(1+x)^\alpha,\,\alpha\in (\max \{0,\frac{2-r}{2}\},\frac12)$, which yields uniform estimates with respect to $N$. Our argument is motivated by the concavity-based Lyapunov approach employed in \cite{brzezniak2005stochastic,chen2025well,tang2024strong,tang2023noise}. However, the particular choice of \(\Phi\) used here is adapted to the structure of the present system and differs from the Lyapunov functions employed in those works.  By combining the specific structure of the nonlocal noise with the negativity of the second derivative of an appropriate concave function, we derive a negative trace term that counteracts the aggregation effect induced by chemotaxis. This leads to estimates for the local strong solutions that are uniform in $N$.
    Furthermore, an application of Markov's inequality shows that the corresponding stopping times tend to infinity almost surely, so that the maximal local solution extends globally in time.

 \section{Preliminaries and the main result}
 In this section, we introduce the probabilistic framework used throughout the paper. We first recall some notation and  basic properties of the cylindrical Wiener processes and the corresponding stochastic integral;  see \cite[Chapter 2]{liu2015stochastic} for details.  We then state the assumptions on the noise coefficient $\sigma$, derive several auxiliary estimates that will be used in the subsequent analysis, and present the definitions of local and maximal solutions. Finally, we formulate the main global existence result. 

 \subsection{Stochastic framework}  
   Let $(U,\langle\cdot,\cdot\rangle_U)$   be a separable Hilbert space with an orthonormal basis $(e_k)_{k\in\mathbb{N}}$.  The space of all bounded linear operators from $U$ to $L^2(\mathcal {O})$ is denoted by $L(U,L^2(\mathcal {O}))$. The space of Hilbert-Schmidt operators from $U$ to $L^2(\mathcal {O})$ is defined by 
 \begin{align*}
     L_2(U,L^2(\mathcal {O}))=\Big\{B\in L(U,L^2(\mathcal {O})):\sum _{k=1}^{\infty}\|Be_k\|_{L^2(\mathcal {O})}^2<\infty\Big\}
 \end{align*}
 and it is endowed with the norm $$\|B\|_{L_2(U,{L^2(\mathcal {O})})}^2=\sum _{k=1}^{\infty}\|Be_k\|_{L^2(\mathcal {O})}^2.$$ 
 
 Let   $(\Omega,\mathcal F,\{\mathcal F_t\}_{t\ge0},\mathbb P)$ be a probability space  with a complete, right-continuous filtration $\{\mathcal F_t\}_{t\ge0}$. Let $\{W_k(t)\}_{k\in\mathbb N}$  be a sequence of independent real-valued $\{\mathcal F_t\}_{t\ge 0}$-Brownian motions. The cylindrical Wiener process on $U$ is formally given by
	\[
	W(t):=\sum_{k=1}^\infty e_k\,W_k(t),\qquad t\in[0,T].
	\]
	We introduce the space of predictable square integrable processes
    \begin{align*}
        &\mathcal N_W^2(0,T;{L^2(\mathcal {O})})\\
	:=&\Big\{\Phi:\Omega\times[0,T]\to L_2(U,{L^2(\mathcal {O})})\ \Big|\ 
	\Phi \text{ is predictable and }
\mathbb E\int_0^T \|\Phi(s)\|_{L_2(U,L^2(\mathcal O))}^2\,ds<\infty
	\Big\}.
    \end{align*}
	For any $\Phi\in\mathcal N_W^2(0,T;{L^2(\mathcal {O})})$ and $t\in[0,T]$, the It\^o integral $\displaystyle\int_0^t \Phi(s)dW(s)$ is well defined and given by
	\[
	\int_0^t \Phi(s)dW(s)
	=\sum_{k=1}^\infty \int_0^t \Phi(s)(e_k)dW_k(s),
	\quad \mathbb P\text{-a.s.},
	\]
	where the series converges in $L^2(\Omega;C([0,T];{L^2(\mathcal {O})}))$. Moreover, the Burkholder-Davis-Gundy inequality yields (see, e.g., \cite[Lemma 3.1]{gawarecki2010stochastic}
	\begin{equation}\label{BDG}
		\mathbb E\Big(\sup_{t\in[0,T]}\Big\|\int_0^t \Phi(s)dW(s)\Big\|_{L^2(\mathcal {O})}^{2q}\Big)
		\le C_p\mathbb E\Big[\Big(\int_0^T \|\Phi(s)\|_{L_2(U,{L^2(\mathcal {O})})}^2\,ds\Big)^q\Big],\quad\forall q\ge1
	\end{equation}
	for some constant $C_q>0$ depending only on $q$. 
	
 \subsection{Assumptions on the noise coefficient} With the stochastic setting in place, we now specify the assumptions and  some elementary properties on the noise coefficient $\sigma$ that will be imposed throughout the paper.

\begin{assumption}\label{ass}
Assume that
\[
2< p<\infty,
\qquad
1<r<\frac{p}{p-2}.
\]
Let $\sigma:H^1(\mathcal O)\to L_2(U,L^2(\mathcal O))$
be defined by
\[
(\sigma(u)e_k)(x):=b_k\|u\|_{L^p(\mathcal O)}^{r}u(x),
\qquad u\in H^1(\mathcal O),\ k\in\mathbb N,
\]
where $(b_k)_{k\ge1}\in \ell^2(\mathbb R).$
\end{assumption}
\begin{remark}
Set $B_0^2:=\sum_{k=1}^\infty b_k^2<\infty.$
Then, for every $u\in H^1(\mathcal O)$,
\[
\|\sigma(u)\|_{L_2(U,L^2(\mathcal O))}^2
=
B_0^2\|u\|_{L^p(\mathcal O)}^{2r}\|u\|_{L^2(\mathcal O)}^2.
\]
Hence $\sigma$ is well defined as an $L_2(U,L^2(\mathcal O))$-valued mapping. Under Assumption \ref{ass}, whenever $\sigma(u)\in \mathcal N_W^2(0,T;L^2(\mathcal O))$,  it follows that 
\begin{align*}
    \int_0^t\sigma(u(s))dW(s)=\sum_{k=1}^\infty b_k\int_0^t\|u(s)\|_{L^p(\mathcal O)}^{r}u(s)dW_k(s).
\end{align*}
\end{remark}

\begin{lem}\label{lem_sigma}
Under Assumption \ref{ass}, there exists a constant $C=C(r,B_0)>0$ such that for all
$u,v\in H^1(\mathcal O)$,
\begin{align}\label{sigma_1}
    \|\sigma(u)-\sigma(v)\|_{L_2(U,L^2(\mathcal O))}^2
\le
C\Big(\|u\|_{L^p(\mathcal O)}^{2r}+\|v\|_{L^p(\mathcal O)}^{2r}\Big)
\Big(
\|u-v\|_{L^2(\mathcal O)}^2+\|u-v\|_{L^p(\mathcal O)}^2
\Big).
\end{align}
Moreover, for every $\varepsilon,\eta>0$, there exists a positive constant $C_{\varepsilon,\eta}(u,v)$
depending on $\|u\|_{L^2(\mathcal O)}$, $\|\nabla u\|_{L^2(\mathcal O)}$,
$\|v\|_{L^2(\mathcal O)}$ and $\|\nabla v\|_{L^2(\mathcal O)}$ such that
\begin{align}\label{sigma_2}
    \|\sigma(u)-\sigma(v)\|_{L_2(U,L^2(\mathcal O))}^2
\le
C_{\varepsilon,\eta }(u,v)\|u-v\|_{L^2(\mathcal O)}^2
+
\varepsilon\|\nabla(u-v)\|_{L^2(\mathcal O)}^2.
\end{align}
\end{lem}
\begin{proof}
By definition of the Hilbert--Schmidt norm,
\begin{align}\label{sigma_3}
\nonumber&\|\sigma(u)-\sigma(v)\|_{L_2(U,L^2(\mathcal O))}^2
\\\nonumber=&
\sum_{k=1}^\infty \|(\sigma(u)-\sigma(v))e_k\|_{L^2(\mathcal O)}^2=
B_0^2
\big\|\|u\|_{L^p(\mathcal O)}^r u-\|v\|_{L^p(\mathcal O)}^r v\big\|_{L^2(\mathcal O)}^2\\
\le &2B_0^2\|u\|_{L^p}^{2r}\|u-v\|_{L^2}^2
+
2B_0^2\big|\|u\|_{L^p}^r-\|v\|_{L^p}^r\big|^2\|v\|_{L^2}^2.
\end{align}
Next, we introduce the following elementary inequalities, for any $a,b\ge0$ and $r\ge1$,
\begin{align}\label{elementary}
    \big|a^r-b^r\big|\le C_r\big(a^{r-1}+b^{r-1}\big)\big|a-b\big|.
\end{align}
Since $p>2$, we have $\|v\|_{L^2(\mathcal O)}\le C\|v\|_{L^p(\mathcal O)}$. Hence, using   (\ref{elementary}), we obtain 
\begin{align}\label{sigma_4}
    \|v\|_{L^2(\mathcal O)}^2
\big|\|u\|_{L^p(\mathcal O)}^r-\|v\|_{L^p(\mathcal O)}^r\big|^2
\le
C \|v\|_{L^2(\mathcal O)}^2\big(\|u\|_{L^p(\mathcal O)}^{2r-2}+\|v\|_{L^p(\mathcal O)}^{2r-2}\big)\|u-v\|_{L^p(\mathcal O)}^2.
\end{align}
Substituting (\ref{sigma_4}) into (\ref{sigma_3}), we obtain
\begin{align}\label{sigma_5}
    \nonumber&\|\sigma(u)-\sigma(v)\|_{L_2(U,L^2(\mathcal O))}^2
\\\le&
C\|u\|_{L^p(\mathcal O)}^{2r}\|u-v\|_{L^2(\mathcal O)}^2+C\Big(\|v\|_{L^2(\mathcal O)}^2\|u\|_{L^p(\mathcal O)}^{2r-2}+\|v\|_{L^2(\mathcal O)}^2\|v\|_{L^p(\mathcal O)}^{2r-2}\Big)
\|u-v\|_{L^p(\mathcal O)}^2.
\end{align}
Therefore, by Young's inequality for \eqref{sigma_5}, we obtain \eqref{sigma_1}.

Applying the Gagliardo-Nirenberg inequality yields
\[
\|u-v\|_{L^p(\mathcal O)}^2
\le
C\|\nabla(u-v)\|_{L^2(\mathcal O)}^{2(1-\frac2p)}
\|u-v\|_{L^2(\mathcal O)}^{\frac4p}+C\|u-v\|_{L^2(\mathcal O)}^2,\qquad2< p<\infty.
\]
Then, using the embedding  Young's inequality, for every $\varepsilon>0$, we have
\begin{align*}
  & C\Big(\|v\|_{L^2(\mathcal O)}^2\|u\|_{L^p(\mathcal O)}^{2r-2}+\|v\|_{L^2(\mathcal O)}^2\|v\|_{L^p(\mathcal O)}^{2r-2}\Big)
\|u-v\|_{L^p(\mathcal O)}^2\\
\le&
\varepsilon \|\nabla(u-v)\|_{L^2(\mathcal O)}^2
+
C_\varepsilon \Big(\|v\|_{L^2(\mathcal O)}^2\|u\|_{L^p(\mathcal O)}^{2r-2}+\|v\|_{L^2(\mathcal O)}^2\|v\|_{L^p(\mathcal O)}^{2r-2}\Big)^\frac{p}{2} \|u-v\|_{L^2(\mathcal O)}^2\\&+C\Big(\|v\|_{L^2(\mathcal O)}^2\|u\|_{L^p(\mathcal O)}^{2r-2}+\|v\|_{L^2(\mathcal O)}^2\|v\|_{L^p(\mathcal O)}^{2r-2}\Big)\|u-v\|_{L^2(\mathcal O)}^2. 
\end{align*}
Hence, combining $r<\frac{p}{p-2}$ and Young's inequality,  we infer 
\begin{align*}
&\|\sigma(u)-\sigma(v)\|_{L_2(U,L^2(\mathcal O))}^2\\
\le&
C_\varepsilon\Big(\|v\|_{L^2(\mathcal O)}^2\|u\|_{L^p(\mathcal O)}^{2r-2}+\|v\|_{L^2(\mathcal O)}^2\|v\|_{L^p(\mathcal O)}^{2r-2}\Big)^\frac{p}{2}
\|u-v\|_{L^2(\mathcal O)}^2\\
&
+\varepsilon \|\nabla(u-v)\|_{L^2(\mathcal O)}^2+C\Big(\|u\|_{L^p(\mathcal O)}^{2r}+\|v\|_{L^p(\mathcal O)}^{2r}\Big)
\|u-v\|_{L^2(\mathcal O)}^2
\\
\le & \varepsilon \|\nabla(u-v)\|_{L^2(\mathcal O)}^2+C_\varepsilon \Big(\|v\|_2^{\frac{2rp}{2r-rp+p}}+\|u\|_{L^p(\mathcal O)}^{2r}+\|v\|_{L^p(\mathcal O)}^{2r}\Big)
\|u-v\|_{L^2(\mathcal O)}^2.
\end{align*}
Applying the Gagliardo-Nirenberg inequality again to $\|u\|_{L^p(\mathcal O)}^{2r}$, we have 
\[
\|u\|_{L^p(\mathcal O)}^{2r}
\le
C\|\nabla u\|_{L^2(\mathcal O)}^{2r(1-\frac2p)}
\|u\|_{L^2(\mathcal O)}^{\frac{4r}{p}}+C\|u\|_{L^2(\mathcal{O})}^{2r}.
\]
Since $r<\frac{p}{p-2}$, we know $r\left(1-\frac{2}{p}\right)<1$. Hence,  by Young's inequality, for every $\eta>0$
\begin{align}\label{G-N}
    \|u\|_{L^p(\mathcal O)}^{2r}
\le
\eta\|\nabla u\|_{L^2(\mathcal O)}^2
+
C_\eta
\|u\|_{L^2(\mathcal O)}^{\frac{4r}{p-r(p-2)}}
+
C\|u\|_{L^2(\mathcal O)}^{2r}.
\end{align}
The same estimate holds for $\|v\|_{L^p(\mathcal O)}^{2r}$. Then, there exists a positive constant
\begin{align}\label{const}
\nonumber C_{\varepsilon,\eta }(u,v)
:=&
C_{\varepsilon,\eta }\big(
\|u\|_{L^2(\mathcal O)}^{\frac{4r}{p-r(p-2)}}
+
\|u\|_{L^2(\mathcal O)}^{2r}+\|v\|_{L^2(\mathcal O)}^{\frac{4r}{p-r(p-2)}}
+
\|v\|_{L^2(\mathcal O)}^{2r}+\|v\|_2^{\frac{2rp}{2r-rp+p}}\big)  \\&+\eta\|\nabla u\|_{L^2(\mathcal O)}^2+
\eta\|\nabla v\|_{L^2(\mathcal O)}^2,
\end{align}
such that
\[
\|\sigma(u)-\sigma(v)\|_{L_2(U,L^2(\mathcal O))}^2
\le
C_{\varepsilon,\eta }(u,v)\|u-v\|_{L^2(\mathcal O)}^2
+
\varepsilon
\|\nabla(u-v)\|_{L^2(\mathcal O)}^2.
\]
This completes the proof.
\end{proof}

\subsection{The initial-boundary value problem} We next establish a few basic estimates for the solution to the following inhomogeneous linear heat equation, which will be used repeatedly in the later analysis
\begin{equation}\label{equ-v_t}
\begin{cases}
v_t=\Delta v-v+g, & x\in\mathcal {O},\ t>0, \\\frac{\partial v}{\partial \nu} = 0,&x\in\partial\mathcal{O},t> 0,\\
v(x,0)=v_0(x), & x\in\mathcal {O}.
\end{cases}
\end{equation}
\begin{lem}\label{lem_v}
	Let $\mathcal O\subset\mathbb R^2$ be a bounded smooth domain and let $v$ be the mild solution of \eqref{equ-v_t}.
    	\begin{itemize}
		\item[(i)] Assume that $v_0\in W^{1,\infty}(\mathcal O)$ and
		$g\in L^2(0,T;H^1(\mathcal O))$. Then there exists a constant $C>0$, independent of $T$, such that for all $t\in[0,T]$,
\begin{equation}\label{v_infty}
			\int_0^t \|\nabla v(s)\|_{L^\infty(\mathcal O)}^2\,ds
			\le
			C\left(
			t\|\nabla v_0\|_{L^\infty(\mathcal O)}^2
			+(t+t^2)
			\int_0^t \|g(s)\|_{H^1(\mathcal O)}^2\,ds
			\right).
		\end{equation}
		\item[(ii)] Assume that
		\[
		1\le m\le q\le \infty,\qquad 1\le \gamma<\infty,
		\qquad \frac1q>\frac1m-\frac12,
		\]
		and suppose that $v_0\in W^{1,q}(\mathcal O), \,g\in L^\gamma(0,T;L^m(\mathcal O)).$
		Then, for all $t\in[0,T]$, there exists a positive constant $C$,
		independent of $T$, such that
	\begin{equation}\label{v_r,q}
			\int_0^t \|\nabla v(s)\|_{L^q(\mathcal O)}^\gamma \,ds
			\le
			C\left(
			t\|\nabla v_0\|_{L^q(\mathcal O)}^\gamma
			+
			\int_0^t \|g(s)\|_{L^m(\mathcal O)}^\gamma \,ds
			\right).
		\end{equation}
	\end{itemize}
\end{lem}
\begin{proof}
	 Let $e^{t\Delta}$ be the Neumann heat semigroup on $\mathcal O$.
	By the mild representation formula,
	 we obtain
	\[
	\nabla v(t)
	=
	e^{-t}\nabla e^{t\Delta}v_0
	+
	\int_0^t e^{-(t-s)}\nabla e^{(t-s)\Delta}g(s)\,ds.
	\]
	Hence, we have 
	\begin{equation}\label{grad-v-infty}
		\|\nabla v(t)\|_{L^q(\mathcal O)}
		\le
		e^{-t}\|\nabla e^{t\Delta}v_0\|_{L^q(\mathcal O)}
		+
		\int_0^t e^{-(t-s)}
		\|\nabla e^{(t-s)\Delta}g(s)\|_{L^q(\mathcal O)}\,ds.
	\end{equation}
    By the standard boundedness of the Neumann heat semigroup on $W^{1,q}(\mathcal O)$, we have
	\begin{align}\label{v_0_q}
	    \|\nabla e^{t\Delta}v_0\|_{L^q(\mathcal O)}
	\le
	C\|\nabla v_0\|_{L^q(\mathcal O)},
	\qquad t>0.
	\end{align}
    Moreover, by the standard  estimate \cite [Lemma 1.3]{winkler2010aggregation} for the Neumann heat semigroup,
    \begin{align}
        \label{gad_infty}&\|\nabla e^{(t-s)\Delta}g(s)\|_{L^\infty(\mathcal O)}
	\le
	C\Bigl(1+(t-s)^{-1/2}\Bigr)\|\nabla g(s)\|_{L^2(\mathcal O)};\\
    \label{gad_q}&\|\nabla e^{(t-s)\Delta}g(s)\|_{L^q(\mathcal O)}
	\le
	C\Bigl(1+(t-s)^{-\alpha}\Bigr)\|g(s)\|_{L^m(\mathcal O)},
    \end{align}
	where
	\[
	\alpha:=\frac12+\frac d2\left(\frac1m-\frac1q\right)
	=
	\frac12+\frac1m-\frac1q
	\qquad \text{since } d=2.
	\]
    
	We first prove (i). Taking $q=\infty$ in \eqref{grad-v-infty} and \eqref{v_0_q}.
	Substituting  estimates (\ref{v_0_q}) and (\ref{gad_infty})  into \eqref{grad-v-infty}, we infer that
	\[
	\|\nabla v(t)\|_{L^\infty(\mathcal O)}
	\le
	C\|\nabla v_0\|_{L^\infty(\mathcal O)}
	+
	C\int_0^t e^{-(t-s)}
	\big(1+(t-s)^{-1/2}\big)\|\nabla g(s)\|_{L^2(\mathcal O)}\,ds.
	\]
	Applying Young's convolution inequality in time yields
	\[
	\|\nabla v\|_{L^2(0,t;L^\infty(\mathcal O))}
	\le
	Ct^{1/2}\|\nabla v_0\|_{L^\infty(\mathcal O)}
	+
C(t+t^{1/2})\|g\|_{L^2(0,t;H^1(\mathcal O))}.
	\]
	Squaring both sides, we obtain
	\[
	\int_0^t \|\nabla v(s)\|_{L^\infty(\mathcal O)}^2\,ds
	\le
	C\left(
	t\|\nabla v_0\|_{L^\infty(\mathcal O)}^2
	+
	(t+t^2)\int_0^t \|g(s)\|_{H^1(\mathcal O)}^2\,ds
	\right),
	\]
	which proves \eqref{v_infty}.
    
		We now prove (ii). Similarly, 
	substituting \eqref{v_0_q} and \eqref{gad_q}  into \eqref{grad-v-infty}, we infer that
	\begin{align*}
	    \|\nabla v(t)\|_{L^q(\mathcal O)}
	\le
	C\|\nabla v_0\|_{L^q(\mathcal O)}
	+
	C\int_0^t e^{-(t-s)}
	\big(1+(t-s)^{-\alpha}\big)\|g(s)\|_{L^m(\mathcal O)}\,ds.
	\end{align*}
	Since $\frac1q>\frac1m-\frac12$,
	we have $\alpha<1$, and hence,  by Young's convolution inequality in time, it follows that
	\[
	\|\nabla v\|_{L^\gamma(0,t;L^q(\mathcal O))}
	\le
	Ct^\frac{1}{\gamma}\|\nabla v_0\|_{L^q(\mathcal O)}
	+
	C\|g\|_{L^\gamma(0,t;L^m(\mathcal O))},
	\]
	which proves \eqref{v_r,q}.
\end{proof}
\subsection{The main theorem }
We now state the main result of this paper. To this end, we first introduce the notions of local and maximal strong solutions to  \eqref{equ1}. In the present paper, the strong solution is understood in the probabilistic sense, that is, the solution is constructed on a given stochastic basis.

  \begin{defn}(Local strong solution)\label{def_local} Fix a stochastic basis
$(\Omega,\mathcal F,\{\mathcal F_t\}_{t\ge0},\mathbb P)$
with a complete and right-continuous filtration, and let $W$ be an $\{\mathcal F_t\}_{t\ge0}$-adapted cylindrical Wiener
process on $U$. Let $u_0$ be an $L^2(\mathcal{O})$-valued  $\mathcal{F}_0$-measurable random variable. A local strong solution to  the equation (\ref{equ1}) is a pair $(u,\tau)$, where $\tau$ is an $\{\mathcal F_t\}$-stopping time  and $u$ is an $L^2(\mathcal O)$-valued continuous adapted process such that for every $T>0$,
	\[
	u(\cdot\wedge\tau)\in L^2(\Omega;C([0,T];L^2(\mathcal O)))
	\cap L^2(\Omega; L^2(0,T;H^1(\mathcal O))).
	\]
    For such a pair $(u,\tau)$, let $v$ be the mild solution of the second
equation in \eqref{equ1} on $[0,\tau]$ associated with $u$ and
$v_0\in W^{1,\infty}(\mathcal O)$ and  for every $T>0$,
\[
v(\cdot\wedge\tau)
\in L^2(0,T;W^{1,\infty}(\mathcal O)),
\qquad \mathbb P\text{-a.s.}
\]
  	Then, for every $\phi\in H^1(\mathcal O)$, the following identity holds
for all $t\in[0,T]$, $\mathbb P$-a.s.
  \begin{align*}
      \big(u(t\wedge{\tau}),\phi \big)_{L^2(\mathcal{O})}=&\big(u_0,\phi \big)_{L^2(\mathcal{O})}-\int_0^{t\wedge{\tau}}\big(\nabla u(s)-u(s)\nabla v(s),\nabla\phi\big)_{L^2(\mathcal{O})} \,ds\\&+\Big(\int_0^{t\wedge{\tau}}\sigma(u(s))dW(s),\phi\Big)_{L^2(\mathcal{O})}.
  \end{align*}

 \end{defn}
   \begin{defn}(Maximal solution)\label{def_max}
A maximal solution to \eqref{equ1} is a pair $(u,\tau^*)$ for which there exists an increasing sequence of stopping times $\{\tau_N\}_{N\in\mathbb N}$ such that $\tau_N\uparrow \tau^*
,\,\mathbb P\text{-a.s.},$
and for each $N\in\mathbb N$, $(u,\tau_N)$ is a local strong solution in the sense of Definition~\ref{def_local}. In addition,
	\[
	\limsup_{t\uparrow\tau^*}
	\Big(
	\sup_{0\le s\le t}
	\|u(s)\|_{L^2(\mathcal O)}^2
	+
	\int_0^t
	\|\nabla u(s)\|_{L^2(\mathcal O)}^2\,ds
	\Big)
	=\infty
	\quad\text{on }\{\tau^*<\infty\},
	\qquad\mathbb P\text{-a.s.}
	\]
If $\tau^*=\infty$ almost surely, then $(u,\tau^*)$ is called a global solution.
\end{defn}
 \begin{thm}\label{thm1}(Global existence and pathwise uniqueness) 
     Under Assumption \ref{ass},  for any $u_0\in L^2(\mathcal {O})$ and  $v_0\in W^{1,\infty}(\mathcal {O}) $,  there exists a global pathwise unique strong solution $u$ of system (\ref{equ1}) in the sense  of Definition \ref{def_max}. 
\end{thm}  
\section{ Global Existence and Pathwise Uniqueness of Solutions}
In this section, we establish   the global existence and pathwise uniqueness of solutions to complete the proof of Theorem~\ref{thm1}. The argument is divided into three steps.  First, we establish the global existence of solutions to the
truncated system \eqref{equ-u} and then derive the local existence of solutions to the
original system \eqref{equ1}.  Second, we prove  pathwise uniqueness for local strong  solutions. Third, relying on the corresponding uniform estimates, we construct a maximal local strong solution to the
 system \eqref{equ1} and then show the global existence result.

\subsection{Local existence} In this subsection, we establish the existence of local strong solutions to system \eqref{equ1}.  For each $N\in\mathbb N$, we first introduce a corresponding truncated system \eqref{equ-u} and
establish its global existence. We then define a suitable stopping time \(\tau_N\) such that the truncated system \eqref{equ-u} coincides with the original system \eqref{equ1}. Consequently, the global strong solution  of the truncated system, when restricted to \([0,\tau_N)\), becomes a local strong solution of the original system \eqref{equ1}. The precise result is stated in the following theorem.
\begin{thm}(Local existence)\label{thm_Local} 
    Suppose that  Assumption \ref{ass} holds. Let $u_0\in L^2(\mathcal {O})$ and  $v_0\in W^{1,\infty}(\mathcal {O}) $, 	 the system (\ref{equ1}) admits  local strong  solutions    in the  sense of Definition \eqref{def_local}.
\end{thm}
\begin{proof}
To construct a local strong solution of system \eqref{equ1}, we first fix
\(N\in\mathbb N\) and use Banach's fixed point theorem to establish the
global existence of a solution \(u_N\) to the truncated system \eqref{equ-u}.
Let $S_T$ be the space of all progressively measurable processes $\rho$
such that 
\begin{align*}
   \rho\in L^2\big(\Omega;C([0,T];L^2(\mathcal O))\big)
\cap
L^2\big(\Omega;L^2(0,T;H^1(\mathcal O))\big), 
\end{align*}
equipped with the norm
\begin{align*}
  \|\rho\|_{S_T}^2
:=\mathbb{E}\big(\|\rho\|_{X_T}^2\big)=
\mathbb E\Big(
\sup_{t\in[0,T]}\|\rho(t)\|_{L^2(\mathcal O)}^2
+
\int_0^T\|\nabla \rho (t)\|_{L^2(\mathcal O)}^2dt\Big).  
\end{align*}
Then $S_T$ is a Banach space. 
	For \(\xi\in S_T\), let \(v[\xi]\) be the progressively measurable process
		such that, for a.e. \(\omega\in\Omega\), \(v[\xi](\omega)\) is the unique
		solution of
		\begin{equation}\label{equ-v}
			\begin{cases}
				\partial_t v=\Delta v-v+\xi(\omega,t,x),
				& x\in\mathcal O,\ t>0,\\
				\dfrac{\partial v}{\partial\nu}=0,
				& x\in\partial\mathcal O,\ t>0,\\
				v(x,0)=v_0(x),
				& x\in\mathcal O.
			\end{cases}
		\end{equation}
        Furthermore, by Lemma \ref{lem_v},  for a.e. $\omega$,  \[
\nabla v[\xi](\omega,\cdot,\cdot)\in L^2(0,T;L^\infty (\mathcal {O})),
\] and the corresponding estimate (\ref{v_infty}) and (\ref{v_r,q}) hold.
For any $\xi\in S_T$, we define $\mathcal T_N(\xi)=u_N$, where $u_N$ is the solution of the following  equation under homogeneous Neumann boundary conditions
\begin{equation}\label{equ-u_1}
\begin{cases}
du_N=\Delta u_Ndt-\theta_N(\|\xi\|_{X_t}^2)\nabla\cdot\big(\xi\nabla v[\xi]\big)dt+\theta_N(\|\xi\|_{X_t}^2)\sigma(\xi)dW(t),\\
u_N(x,0)=u_0.
\end{cases}
\end{equation}
Recall that  $\theta_N$ is the same cutoff function as that appearing in
\eqref{equ-u}. Hence,  for each fixed $N\in\mathbb N$ and for a.e. $\omega\in\Omega$, if $\theta_N(\|\xi\|_{X_t}^2)\neq0$,  we necessarily have $$\sup_{s\in[0,t]}\|\xi(s)\|_{L^2(\mathcal {O})}^2+\int_0^t\|\nabla \xi(s)\|_{L^2(\mathcal {O})}^2\,ds< 2N.$$
Thus the truncation guarantees that the chemotactic drift and noise coefficient is active only when $\xi $ remains bounded in $X_t$. 
 We shall prove that for sufficiently small $T>0$, the map $$\mathcal{T}_N:S_T\to S_T$$ is a contraction. 
 
 As a first step, we show $\mathcal{T}_N$ is well defined. Applying the existence theory of stochastic differential equations (\cite[Theorem 4.2.4]{ liu2015stochastic}) to prove that for every $\xi\in S_T$, the equation (\ref{equ-u_1}) admits a unique solution $u_N\in S_T$.

  Fix a stochastic basis
$(\Omega,\mathcal F,\{\mathcal F_t\}_{t\ge0},\mathbb P)$
with a complete and right-continuous filtration, and let $W$ be an $\{\mathcal F_t\}_{t\ge0}$-adapted cylindrical Wiener
process on $U$. We work with the Gelfand triple $H^1(\mathcal {O})\hookrightarrow L^2(\mathcal {O})\hookrightarrow H^{-1}(\mathcal {O}).$
 Set \[
 A(t,u_N,\omega):=\Delta u_N-\nabla\cdot f_\xi(t,\omega),
 \]
 where
 $$
f_\xi(t,\omega):=\theta_N(\|\xi(\omega)\|_{X_t}^2)\Big(\xi(t,\omega)\nabla v[\xi](t,\omega)\Big).$$
 For  fixed $\xi\in S_T$, using H\"older's inequality, we have 
\begin{align}\label{f_xi}
\nonumber&\int_0^T \|\nabla\cdot f_\xi(t,\omega)\|_{H^{-1}(\mathcal {O})}^2dt\le C\int_0^T \| f_\xi(t,\omega)\|_{L^2(\mathcal {O})}^2dt \\
\nonumber\le&
C\int_0^T \big|\theta_N(\|\xi\|_{X_t}^2)\big|^2
\|\xi(t,\omega)\|_{L^2(\mathcal {O})}^2
\|\nabla v[\xi](t,\omega)\|_{L^\infty(\mathcal {O})}^2dt\\
\le&C\sup_{t\in[0,T]}\Big(\big|\theta_N(\|\xi\|_{X_t}^2)\big|^2\|\xi(t,\omega)\|_{L^2(\mathcal {O})}^2\Big)
\int_0^T \|\nabla v[\xi](t,\omega)\|_{L^\infty(\mathcal {O})}^2dt.
\end{align}
Combining with the definition of the cutoff function $\theta_N(\|\xi(\omega)\|_{X_t}^2)$ and   the estimate \eqref{v_infty} of $v[\xi]$ to \eqref{equ-v}, 
we infer that for a.e. $\omega\in\Omega$,
\begin{align}\label{f_xi_2}
    \int_0^T \|\nabla\cdot f_\xi(t,\omega)\|_{H^{-1}(\mathcal {O})}^2dt\le C\int_0^T \| f_\xi(t,\omega)\|_{L^2(\mathcal {O})}^2dt
\le
C(N,T,v_0).
\end{align}
Modifying it on a
	\(dt\otimes d\mathbb P\)-null set if necessary, we may assume that
	\[
	f_\xi:[0,T]\times\Omega\to L^2(\mathcal O).
	\]
	 This modification does not affect any of the
	Bochner integrals appearing below.
Moreover,  since $\xi$ is progressively measurable and $v[\xi]$ is obtained from $\xi$ through the linear parabolic equation \eqref{equ-v}, the process $v[\xi]$ is also progressively measurable. Hence, the operator $$A:[0,T]\times H^1(\mathcal {O})\times\Omega\to H^{-1}(\mathcal {O})$$ is well-defined and for each $u_N\in  H^1(\mathcal {O})$, the map $(t,\omega)\mapsto A(t,u_N,\omega) $ is progressively measurable. Set
\[
B(t,u_N(t),\omega):=\theta_N(\|\xi(\omega)\|_{X_t}^2)\sigma(\xi(t,\omega)),
\]
where $\sigma(\xi)$ satisfies Assumption \ref{ass}. For a fixed $\xi\in S_T$, by Lemma \ref{lem_sigma}, we have 
\begin{align}\label{B}
\|B(t,u_N(t))\|_{L_2(U,L^2(\mathcal{O}))}^2=\big|\theta_N(\|\xi\|_{X_t}^2)\big|^2\|\sigma(\xi)\|_{L_2(U,L^2(\mathcal{O}))}^2< \infty
\end{align}
for $dt\otimes d\mathbb{P}$-a.e. $(t,\omega)$.
Thus, after modifying it on $dt\otimes d\mathbb{P}$-null set,  the operator $$B:[0,T]\times H^1(\mathcal {O})\times\Omega\to {L_2(U,L^2(\mathcal{O}))} $$ is well-defined and progressively measurable.

In the following, for each $(t,\omega)\in[0,T]\times\Omega$, we will verify that  $A(t,\cdot,\omega)$, $B(t,\cdot,\omega)$ satisfy all the assumptions required by the variational existence theorem (\cite{ liu2015stochastic,rockner2024well,krylov1981stochastic}).

\noindent\textbf{(H1) (Hemicontinuity).}
Let $u_{N},\bar{u}_N,w\in H^1(\mathcal {O})$. For every $\lambda\in\mathbb{R}$, the map
\[
\lambda\mapsto\big\langle A(t,u_N+\lambda \bar{u}_N,\omega),w\big\rangle_{H^{-1}{(\mathcal {O})},H^1{(\mathcal {O})}}
=
-\int_{\mathcal {O}}\nabla (u_N+\lambda \bar u_N)\cdot \nabla w\,dx+\int_{\mathcal {O}} f_\xi\cdot \nabla w dx
\]
is continuous on $\mathbb{R}$. 

\noindent\textbf{(H2) (Weak monotonicity).} For all $u_N,\bar  u_N\in H^1(\mathcal {O})$, 
\begin{align*}
	&2\big\langle A(t,u_N,\omega)-A(t,\bar 
 u_N,\omega),u_N-\bar u_N\big\rangle_{H^{-1}{(\mathcal {O})},H^1{(\mathcal {O})}}
	\\
	=&2\langle \Delta(u_N-\bar u_N),u_N-\bar u_N\rangle_{H^{-1}(\mathcal {O}),H^1(\mathcal {O})}
	=-2\|\nabla(u_N-\bar u_N)\|_{L^2(\mathcal O)}^2.
\end{align*}

\noindent\textbf{(H3) (Coercivity).}
Let $u_N\in H^1(\mathcal {O})$. By Young's inequality and  Lemma \ref{lem_sigma}, we have
\begin{align*}
    &2\big\langle A(t,u_N,\omega),u_N\big\rangle_{H^{-1}{(\mathcal {O})},H^1{(\mathcal {O})}}+\|B(t,u_N,\omega)\|_{L_2(U,L^2(\mathcal O))}^2\\\le&
-\|\nabla u_N\|_{L^2(\mathcal {O})}^2+2\|f_\xi(t,\omega)\|_{L^2(\mathcal {O})}^2+B_0^2\big|\theta_N(\|\xi\|_{X_t}^2)\big|^2\|\xi\|_{L^p(\mathcal{O})}^{2r}\|\xi\|_{L^2(\mathcal{O})}^2.
\end{align*}
By (\ref{f_xi_2}) and the definition of function $\theta_N$ , the last two terms are integrable on $[0,T]\times\Omega$.

\noindent\textbf{(H4) (Growth).}
For any $u_N,w\in H^1(\mathcal {O})$, we have
\begin{align*}
    \big|\langle A(t,u_N,\omega),w\rangle_{H^{-1}{(\mathcal {O})},H^1{(\mathcal {O})}}\big|
&=
\left|\int_{\mathcal {O}}\nabla u_N\cdot \nabla w\,dx+\int_{\mathcal {O}}f_\xi \cdot \nabla w\,dx\right|
\\&\le
\|\nabla u_N\|_{L^2(\mathcal {O})}\|\nabla w\|_{L^2(\mathcal {O})}+\|f_\xi(t,\omega)\|_{L^2(\mathcal {O})}\|\nabla w\|_{L^2(\mathcal {O})}.
\end{align*}
Taking the supremum over all $w\in H^1(\mathcal {O})$ with $\|w\|_{H^1(\mathcal {O})}\le 1$, we deduce that
\[
\|A(t,u_N,\omega)\|_{H^{-1}(\mathcal {O})}
\le C\big(\|u_N\|_{H^1(\mathcal {O})}+\|f_\xi(t,\omega)\|_{L^2(\mathcal {O})}\big),
\]
where $\|f_\xi\|_{L^2(\mathcal {O})}\in  L^2([0,T]\times \Omega) $.

Thus all the assumptions of the variational existence theorem \cite[Theorem 4.2.4]{liu2015stochastic} and \cite[Theorem 4.7]{gawarecki2010stochastic} are satisfied. Consequently, for each $\xi\in S_T$,
there exists a unique continuous $L^2(\mathcal{O})$-valued ${\mathcal{F}_t}$-adapted process $u_N$
to \eqref{equ-u_1}. In particular,
\[
\mathbb E\Big(\sup_{t\in[0,T]}\|u_N(t)\|_{L^2(\mathcal O)}^2\Big)
+\mathbb E\Big(\int_0^T\|\nabla u_N(t)\|_{L^2(\mathcal O)}^2dt\Big)<\infty,
\]
and hence $u_N\in S_T$. Therefore, $\mathcal T_N:S_T\to S_T$ is well defined.

Next, we prove that there exists $T^*>0$ such that the mapping
\[
\mathcal T_N:S_{T^*}\to S_{T^*}
\]
is a contraction. More precisely, we shall show that
\[
\|\mathcal T_N(\xi)-\mathcal T_N(\bar\xi)\|_{S_{T^*}}
\le q(T^*)\,\|\xi-\bar\xi\|_{S_{T^*}},
\qquad \forall\,\xi,\bar\xi\in S_{T^*},
\]
for some constant $q(T^*)\in(0,1)$.
To this end, let $u_N=\mathcal T_N(\xi), \bar u_N=\mathcal T_N(\bar\xi), v=v[\xi], \bar v=v[\bar\xi].$
Set $w_N:=u_N-\bar u_N$. Then $w_N$ satisfies $w_N(0)=0$ and
    \begin{align*}
       dw_N=&\Big[\Delta {w_N}-\Big(\theta_N(\|\xi\|_{X_t}^2)\nabla\cdot(\xi\nabla v)-\theta_N(\|\bar \xi\|_{X_t}^2)\nabla\cdot(\bar \xi\nabla \bar v)\Big)\Big]dt\\&+\Big(\theta_N(\|\xi\|_{X_t}^2)\sigma(\xi)-\theta_N(\|\bar \xi\|_{X_t}^2)\sigma(\bar \xi )\Big)dW(t).
    \end{align*}
    Applying the It\^o formula to $\|{w_N}(t)\|_{L^2(\mathcal {O})}^2$,   we obtain that for every  $t\in[0,T]$,
    \begin{align}\label{U_1}
    \nonumber&\|{w_N}(t)\|_{L^2(\mathcal {O})}^2\\\nonumber
    =&2\int_0^t\big\langle \Delta {w_N}(s),{w_N}(s)\big\rangle_{{H^{-1}(\mathcal {O})},{H^1(\mathcal {O})}}\,ds\\\nonumber &-2\int_0^t\big\langle\big(\theta_N(\|\xi\|_{X_s}^2)\nabla\cdot(\xi(s)\nabla v(s))-\theta_N(\|\bar \xi\|_{X_s}^2)\nabla\cdot(\bar \xi(s)\nabla \bar v(s))\big),{w_N}(s)\big\rangle_{{H^{-1}(\mathcal {O})},{H^1(\mathcal {O})}}\,ds\\\nonumber &+\int_0^t\|\theta_N(\|\xi\|_{X_s}^2)\sigma(\xi(s))-\theta_N(\|\bar \xi\|_{X_s}^2)\sigma(\bar \xi(s) )\|_{L_2({w_N},L^2(\mathcal {O}))}^2 \,ds\\\nonumber&+2\int_0^t\big\langle {w_N}(s),\big(\theta_N(\|\xi\|_{X_s}^2)\sigma(\xi(s))-\theta_N(\|\bar \xi\|_{X_s}^2)\sigma(\bar \xi(s) )\big)dW(s)\big\rangle_{L^2(\mathcal {O})}\\\nonumber
    \le&-\int_0^t\|\nabla {w_N}(s)\|_{L^2(\mathcal {O})}^2\,ds+\int_0^t\|\theta_N(\|\xi\|_{X_s}^2)(\xi\nabla v)-\theta_N(\|\bar \xi\|_{X_s}^2)(\bar \xi\nabla \bar v)\|_{L^2(\mathcal {O})}^2\,ds\\&\nonumber +B_0^2\int_0^t
\big\|\theta_N(\|\xi\|_{X_s}^2)\|\xi(s)\|_{L^p(\mathcal O)}^r \xi(s)-\theta_N(\|\bar \xi\|_{X_s}^2)\|\bar \xi(s) \|_{L^p(\mathcal O)}^r \bar \xi(s)\big\|_{L^2(\mathcal O)}^2\,ds\\&+2\int_0^t\big\langle {w_N}(s),\big(\theta_N(\|\xi\|_{X_s}^2)\sigma(\xi(s))-\theta_N(\|\bar \xi\|_{X_s}^2)\sigma(\bar \xi(s) )\big)dW(s)\big\rangle_{L^2(\mathcal {O})}.
    \end{align}
     We define \begin{align*}
H(s):=&\|\theta_N(\|\xi\|_{X_s}^2)\xi\nabla v-\theta_N(\|\bar \xi\|_{X_s}^2)\bar \xi\nabla \bar v\|_{L^2(\mathcal {O})}^2\end{align*}
    Set   $A_s:=\big\{\omega\in\Omega:\,\|\xi(\omega)\|_{X_s}^2\vee \|\bar \xi(\omega)\|_{X_s}^2< 2N\big\}$. Since $\theta_N$ is Lipschitz continuous and $\big|\| \xi\|_{X_s}-\| \bar \xi\|_{X_s}\big|\le \|\xi- \bar \xi\|_{X_s}$, we have
\begin{align}
    \nonumber&\big|\theta_N(\|\xi\|_{X_s}^2)-\theta_N(\|\bar \xi\|_{X_s}^2)\big|^2\le \frac{{C^2_\theta}}{N^2}\big|\|\xi\|_{X_s}^2-\|\bar \xi\|_{X_s}^2\big|^2\\\le& \frac{{C^2_\theta}}{N^2}\big(\|\xi\|_{X_s}+\|\bar \xi\|_{X_s}\big)^2\big(\|\xi\|_{X_s}-\|\bar \xi\|_{X_s}\big)^2\le C\|\xi-\bar \xi\|_{X_s}^2 .
\end{align} Thus, 
    \begin{align*}
\int_0^t1_{A_s}H(s)\,ds\le&3\int_0^t1_{A_s}\theta_N(\|\xi\|_{X_s}^2)^2\|\xi-\bar \xi\|_{L^2(\mathcal {O})}^2\|\nabla v\|_{L^\infty(\mathcal {O})}^2\,ds\\&+3\int_0^t1_{A_s}\big|\theta_N(\|\xi\|_{X_s}^2)-\theta_N(\|\bar \xi\|_{X_s}^2)\big|^2\| \bar \xi\|_{L^2(\mathcal {O})}^2\|\nabla v\|_{L^\infty(\mathcal {O})}^2\,ds\\&+3\int_0^t1_{A_s}\theta_N(\|\bar \xi\|_{X_s}^2)^2\|\bar \xi\|_{L^2(\mathcal {O})}^2\|\nabla(v- \bar v)\|_{L^\infty(\mathcal {O})}^2\,ds\\\le&C\big(\sup_{s\in[0,t]}\|\xi(s)-\bar \xi(s)\|_{L^2(\mathcal {O})}^2\big)\int_0^t1_{A_s}\|\nabla v\|_{L^\infty(\mathcal {O})}^2\,ds\\&+C\|\xi- \bar \xi\|_{X_t}^21_{A_s}\big(\sup_{s\in[0,t]}\|\bar \xi(s)\|_{L^2(\mathcal {O})}^2\big)\int_0^t\|\nabla v\|_{L^\infty(\mathcal {O})}^2\,ds\\&+C1_{A_s}\big(\sup_{s\in[0,t]}\|\bar \xi(s)\|_{L^2(\mathcal {O})}^2\big)\int_0^t\|\nabla (v-\bar v)\|_{L^\infty(\mathcal {O})}^2\,ds.
\end{align*}
By Lemma \ref{lem_v} and the definition of the norm $X_t$, we obtain 
\begin{align}\label{1_A}
   \nonumber \int_0^t1_{A_s}H(s)\,ds\le& C(N,v_0)(t+t^2+t^3)\sup_{s\in[0,t]}\|\xi(s)-\bar \xi(s)\|_{L^2(\mathcal {O})}^2\\\nonumber &+C(N,v_0)(t+t^2+t^3)\|\xi-\bar \xi\|_{X_t}^2\\\le&C(N,v_0)(t+t^2+t^3)\|\xi-\bar \xi\|_{X_t}^2.
\end{align}
We define $B_s:=\big\{\omega\in\Omega:\|\xi(\omega)\|_{X_s}^2\wedge \|\bar \xi(\omega)\|_{X_s}^2\ge 2N\big\}$, then
    \begin{align}\label{1_B}
        \int_0^t1_{B_s}H(s)\,ds=0
\end{align}
Let $C_s:=\big\{\omega\in\Omega:\|\xi(\omega)\|_{X_s}^2< 2N, \|\bar \xi(\omega)\|_{X_s}^2\ge 2N\big\}$, then
    \begin{align}\label{1_C}
        \nonumber\int_0^t1_{C_s}H(s)\,ds&=\int_0^t1_{C_s}\Big|\theta_N(\|\xi\|_{X_s}^2)\Big|^2\|\xi\nabla v\|_{L^2(\mathcal {O})}^2\,ds\\
        \nonumber&=\int_0^t1_{C_s}\Big|\theta_N(\|\xi\|_{X_s}^2)-\theta_N(\|\bar \xi\|_{X_s}^2)\Big|^2\|\xi\nabla v\|_{L^2(\mathcal {O})}^2\,ds\\
        \nonumber&\le C(N)\|\xi- \bar \xi\|_{X_t}^2\int_0^t1_{C_s}\|\nabla v\|_{L^\infty(\mathcal {O})}^2\|\xi\|_{L^2(\mathcal {O})}^2\,ds\\&\le C(N,v_0)(t+t^2+t^3)\|\xi-\bar \xi\|_{X_t}^2.
\end{align}
Similarly, for $D_s:=\big\{\omega\in\Omega:\|\xi(\omega)\|_{X_s}^2\ge 2N, \|\bar \xi(\omega)\|_{X_s}^2< 2N\big\}$, we have 
    \begin{align}\label{1_D}
        \int_0^t 1_{D_s}H(s)\,ds\le C(N,v_0)(t+t^2+t^3)\|\xi-\bar \xi\|_{X_t}^2.
\end{align}
Combining the above estimates \eqref{1_A}-\eqref{1_D}, we obtain
\begin{align}\label{H(s)}
\nonumber\int_0^tH(s)\,ds&=\int_0^t1_{A_s}H(s)\,ds+\int_0^t1_{B_s}H(s)\,ds+\int_0^t1_{C_s}H(s)\,ds+\int_0^t1_{D_s}H(s)\,ds\\&\le C(N,v_0)(t+t^2+t^3)\|\xi-\bar \xi\|_{X_t}^2.
\end{align}
 Next, we apply the same method as for $H(s)$ to analyze $G(s)$. 
 
On $A_s$, by Lemma \ref{lem_sigma} on $\sigma$, we  obtain 
\begin{align*}
\int_0^t1_{A_s}G(s)ds=&\int_0^t1_{A_s}\big\|\theta_N(\|\xi\|_{X_s}^2)\|\xi\|_{L^p(\mathcal O)}^r \xi-\theta_N(\|\bar \xi\|_{X_s}^2)\|\bar \xi \|_{L^p(\mathcal O)}^r \bar \xi\big\|_{L^2(\mathcal O)}^2ds
\\
\le &2\int_0^t1_{A_s}\big|\theta_N(\|\xi\|_{X_s}^2)-\theta_N(\|\bar \xi\|_{X_s}^2)\big|^2\|\xi\|_{L^p(\mathcal O)}^{2r}\|\xi\|_{L^2(\mathcal O)}^2ds\\&+2\int_0^t1_{A_s}\theta_N(\|\bar \xi\|_{X_s}^2)^2\big\|\|\xi\|_{L^p(\mathcal O)}^r \xi-\|\bar \xi \|_{L^p(\mathcal O)}^r \bar \xi\big\|_{L^2(\mathcal O)}^2ds\\
\le &2\int_0^t1_{A_s}\big|\theta_N(\|\xi\|_{X_s}^2)-\theta_N(\|\bar \xi\|_{X_s}^2)\big|^2\|\xi\|_{L^p(\mathcal O)}^{2r}\|\xi\|_{L^2(\mathcal O)}^2ds\\&+\int_0^t1_{A_s}C_{\varepsilon,\eta}(\xi,\bar\xi) \|\xi-\bar \xi\|_{L^2(\mathcal O)}^2ds+\varepsilon \int_0^t\|\nabla \xi-\nabla\bar \xi\|_{L^2(\mathcal O)}^2ds,
\end{align*}
where \(C_{\varepsilon,\eta}(\xi,\bar\xi)\) is defined by
\eqref{const}, namely,  \begin{align*}
   C_{\varepsilon,\eta}(\xi,\bar\xi)=&C\big(
\|\xi\|_{L^2(\mathcal O)}^{\frac{4r}{p-r(p-2)}}
+
\|\xi\|_{L^2(\mathcal O)}^{2r}+\|\bar \xi\|_{L^2(\mathcal O)}^{\frac{4r}{p-r(p-2)}}
+
\|\bar \xi\|_{L^2(\mathcal O)}^{2r}+\|\bar \xi\|_2^{\frac{2rp}{2r-rp+p}}\big)  \\&+\eta\|\nabla \xi\|_{L^2(\mathcal O)}^2+
\eta\|\nabla \bar \xi\|_{L^2(\mathcal O)}^2.
\end{align*} 
Furthermore, using \eqref{G-N} and choosing $\eta>0$ sufficiently small 
\begin{align*}
    \int_0^t1_{A_s}G(s)\,ds\le &\varepsilon\|\xi-\bar \xi\|_{X_t}^2
+C(N)
\|\xi-\bar \xi\|_{X_t}^2\int_0^t1_{A_s}\|\xi\|_{L^p(\mathcal{O})}^{2r}ds\\
&+\|\xi-\bar \xi\|_{X_t}^2\int_0^t1_{A_s}  C_{\varepsilon,\eta}(\xi(s),\bar\xi(s))ds\\
\le &C(N,\varepsilon)t\|\xi-\bar \xi\|_{X_t}^2+\varepsilon\|\xi-\bar \xi\|_{X_t}^2.
\end{align*}
On $C_s$, we have 
\begin{align*}
\int_0^t1_{C_s}G(s)\,ds=&\int_0^t1_{C_s}\big|\theta_N(\|\xi\|_{X_s}^2)-\theta_N(\|\bar \xi\|_{X_s}^2)\big|^2\big\|\xi\|_{L^p(\mathcal O)}^r \xi\big\|_{L^2(\mathcal O)}^2ds\\\le& C(N)\|\xi-\bar \xi\|_{X_t}^2\int_0^t1_{C_s}\|\xi\|_{L^p(\mathcal{O})}^{2r}ds\le C(N,\varepsilon)t\|\xi-\bar \xi\|_{X_t}^2+\varepsilon\|\xi-\bar \xi\|_{X_t}^2.
\end{align*}
Therefore, we infer
\begin{align}\label{G(s)}
\nonumber\int_0^tG(s)\,ds&=\int_0^t1_{A_s}G(s)\,ds+\int_0^t1_{B_s}G(s)\,ds+\int_0^t1_{C_s}G(s)\,ds+\int_0^t1_{D_s}G(s)\,ds\\&\le C(N,\varepsilon)t\|\xi-\bar \xi\|_{X_t}^2+\varepsilon\|\xi-\bar \xi\|_{X_t}^2.
\end{align}
Therefore, substituting the above inequality \eqref{H(s)} and \eqref{G(s)} into \eqref{U_1} yields
\begin{align}\label{U_2}
&\nonumber\|{w_N}(t)\|_{L^2(\mathcal {O})}^2+\int_0^t\|\nabla {w_N}(s)\|_{L^2(\mathcal {O})}^2\,ds\\
\nonumber\le &\big(C(N,v_0,\varepsilon)(t+t^2+t^3)+\varepsilon\big)\|\xi-\bar \xi\|_{X_t}^2\\&+2\int_0^t\big\langle
{w_N}(s),\big(\theta_N(\|\xi\|_{X_s}^2)\sigma(\xi(s))-\theta_N(\|\bar \xi\|_{X_s}^2)\sigma(\bar \xi(s) )\big)dW(s)\big\rangle_{L^2(\mathcal {O})}.
\end{align}
  Taking the supremum over \(t\in[0,T]\) in \eqref{U_2} and then taking
expectations, we obtain
    \begin{align}\label{EU}
\nonumber&\mathbb{E}\big(\sup_{t\in[0,T]}\|{w_N}(t)\|_{L^2(\mathcal {O})}^2\big)+\mathbb{E}\Big(\int_0^T\|\nabla {w_N}(t)\|_{L^2(\mathcal {O})}^2dt\Big)\\
    \nonumber\le &\big(C(N,v_0,\varepsilon)(T+T^2+T^3)+\varepsilon\big)\|\xi-\bar \xi\|_{S_T}^2\\
&+2\mathbb{E}\Big(\sup_{t\in[0,T]}\Big|\int_0^t\big\langle {w_N}(s),\big(\theta_N(\|\xi\|_{X_s}^2)\sigma(\xi(s))-\theta_N(\|\bar \xi\|_{X_s}^2)\sigma(\bar \xi(s) )\big)dW(s)\big\rangle_{L^2(\mathcal {O})}\Big|\Big).
    \end{align}
    Using the Burkholder-Davis-Gundy inequality and Young's inequality, combining the estimate \eqref{G(s)} we have
    \begin{align*}
        &2\mathbb{E}\Big(\sup_{t\in[0,T]}\Big|\int_0^t\big\langle {w_N}(s),\big(\theta_N(\|\xi\|_{X_s}^2)\sigma(\xi(s))-\theta_N(\|\bar \xi\|_{X_s}^2)\sigma(\bar \xi(s) )\big)dW(s)\big\rangle_{L^2(\mathcal {O})}\Big|\Big)\\\le&C \mathbb{E}\Big(\int_0^T\|{w_N}(t)\|_{L^2(\mathcal {O})}^2\|\theta_N(\|\xi\|_{X_t}^2)\sigma(\xi(t))-\theta_N(\|\bar \xi\|_{X_t}^2)\sigma(\bar \xi(t) )\|_{L_2(U,{L^2(\mathcal {O})})}^2dt\Big)^\frac{1}{2}\\
        \le &C\mathbb{E}\Big(\sup_{t\in[0,T]}\|{w_N}(t)\|_{L^2(\mathcal {O})}^2\int_0^T\big\|\theta_N(\|\xi\|_{X_t}^2)\|\xi\|_{L^p(\mathcal O)}^r \xi-\theta_N(\|\bar \xi\|_{X_t}^2)\|\bar \xi \|_{L^p(\mathcal O)}^r \bar \xi\big\|_{L^2(\mathcal O)}^2dt\Big)^\frac{1}{2}\\
        \le &\frac{1}{4}\mathbb{E}\Big(\sup_{t\in[0,T]}\|{w_N}(t)\|_{L^2(\mathcal {O})}^2\Big)+\big(C(N,v_0,\varepsilon)T+\varepsilon\big)\mathbb{E}\|\xi-\bar \xi\|_{X_T}^2.
    \end{align*}
   Substituting the above inequality into (\ref{EU}), we obtain
   \begin{align*}
&\mathbb{E}\Big(\sup_{t\in[0,T]}\|{w_N}(t)\|_{L^2(\mathcal {O})}^2\Big)+\mathbb{E}\Big(\int_0^T\|\nabla {w_N}(t)\|_{L^2(\mathcal {O})}^2dt\Big)\\
\le& \big(C(N,v_0,\varepsilon)(T+T^2+T^3)+\varepsilon\big)\|\xi-\bar \xi\|_{S_T}^2.
    \end{align*}
Equivalently,
\[
\|\mathcal T_N(\xi)-\mathcal T_N(\bar\xi)\|_{S_T}
\le q(T)\|\xi-\bar \xi\|_{S_T},
\]
where $q(T):=\big(C(N,v_0,\varepsilon)(T+T^2+T^3)+\varepsilon\big)^\frac{1}{2}$.Choose $\varepsilon>0$ sufficiently small and then choose $T^*>0$
sufficiently small such that $q(T^*)<1$.
Hence, $\mathcal T_N$ is a contraction on $S_{T^*}$. 

By Banach's fixed point theorem, there exists a unique fixed point $u_N^*\in S_{T^*}$ of $\mathcal T_N$. In particular, $u_N^*$ is the unique solution to the truncated equation \eqref{equ-u} on $[0,T^*]$.

We now extend the local strong solution from $[0,T^*]$ to $[0,2T^*]$. Denote
\[
\widetilde S_{2T^*}
:=
\Big\{
u\in S_{2T^*}:\ u=u_N^*\ \text{on}\,[0,T^*],\ \mathbb P\text{-a.s.}
\Big\}.
\]
Equipped with the metric induced by $S_{2T^*}$,
\[
d(u,\bar u):=\|u-\bar u\|_{S_{2T^*}},
\]
the space $\widetilde S_{2T^*}$ is a complete metric space.

For any $\xi\in \widetilde S_{2T^*}$, let $v[\xi]$ denote the unique solution on $[0,2T^*]$ of the corresponding linear parabolic equation. Since $\xi\in \widetilde S_{2T^*}\subset S_{2T^*}$,    the same estimates from Lemma \ref{lem_v}    remain valid for $v[\xi]$ on $[0,2T^*]$.
We define $\widetilde{\mathcal T}_N(\xi)=u_N$, where $u_N\in \widetilde S_{2T^*}$ is given by
\[
u_N(t)=u_N^*(t),\qquad t\in[0,T^*],
\]
and on $[T^*,2T^*]$, $u_N$ solves
\begin{equation}\label{equ-u2}
\begin{cases}
du_N=\Delta u_N\,dt-\theta_N(\|\xi\|_{X_t}^2)\nabla\cdot\bigl(\xi\nabla v[\xi]\bigr)dt+\theta_N(\|\xi\|_{X_t}^2)\sigma(\xi)dW(t),
& t\in[T^*,2T^*],\\
u_N(T^*)=u_N^*(T^*).
\end{cases}
\end{equation}
Since $u_N^*\in L^2\big(\Omega;C([0,T^*];L^2(\mathcal {O}))\big)$,
the random variable $u_N^*(T^*)$ is well defined and $\mathcal F_{T^*}$-measurable.

Let $\xi,\bar\xi\in\widetilde S_{2T^*}$. Since $\xi=\bar\xi=u_N^*$ on $[0,T^*]$, we have $\xi-\bar\xi=0$ on $[0,T^*]$. Hence, the same difference estimates as in $[0,T^*]$ apply on $[T^*,2T^*]$ with the same contraction constant depending only on the length $T^*$ of the interval. Therefore,
\[
\|\widetilde{\mathcal T}_N(\xi)-\widetilde{\mathcal T}_N(\bar\xi)\|_{ S_{2T^*}}
\le
q(T^*)\,\|\xi-\bar\xi\|_{S_{2T^*}},
\]
where
\[
q(T^*)
=
\big(C(N,v_0,\varepsilon)T^*+\varepsilon\big)^{1/2}<1.
\]
Thus $\widetilde{\mathcal T}_N$ is a contraction on $\widetilde S_{2T^*}$. By Banach's fixed point theorem, there exists a unique fixed point $u_N^{**}\in \widetilde S_{2T^*},$
which is the unique solution to the truncated equation on $[0,2T^*]$. Now choose $m\in\mathbb N$ such that $mT^*\ge T$. Repeating the above argument finitely many times on the intervals $[kT^*,(k+1)T^*]$, $k=1,\dots,m-1$, we obtain a unique solution $u_N\in S_T$ to the truncated equation \eqref{equ-u} on $[0,T]$.

Define the stopping time
      $$\tau_N:=\inf\big\{t\ge 0:\ \sup_{s\in[0,t]}\|u_N(s)\|_{L^2(\mathcal {O})}^2+\int_0^t\|\nabla u_N(s)\|_{L^2(\mathcal {O})}^2\,ds\ge N\big\},$$
By the definition of $\theta_N$, we have
\[
\theta_N(\|u_N\|_{X_t}^2)=1,
\qquad  0\le t <\tau_N.
\]
Therefore, on $[0,\tau_N)$ the truncated equation (\ref{equ-u}) coincides with the original equation (\ref{equ1}). Consequently, $(u_N,\tau_N)$ is a local strong solution to \eqref{equ1} and for every \(T>0\),
\[
u_N(\cdot\wedge\tau_N)
\in L^2(\Omega;C([0,T];L^2(\mathcal O)))
\cap L^2(\Omega;L^2(0,T;H^1(\mathcal O))).
\] 
Thus, we complete the proof of Theorem \ref{thm_Local}.
\end{proof}
\subsection{Pathwise uniqueness and maximal solution}
In this subsection, we establish pathwise uniqueness for local strong solutions  on a fixed stochastic basis. Furthermore, combining the local existence and pathwise uniqueness results, we construct a maximal  solution.

\begin{thm}(Pathwise uniqueness)\label{thm3.2}
      Under Assumption \ref{ass}, let $(u_{N_1},\tau_{N_1})$ and $(u_{N_2},\tau_{N_2})$ be two local  solutions to \eqref{equ1} in the sense of Definition \ref{def_local}, defined on the same stochastic basis with
the same Wiener process. Let $v_1:=v[u_{N_1}], \,  v_2:=v[u_{N_2}]$
and assume that the two corresponding $v$-equations have the
same initial data $v_0$. Assume that
      \[
	u_{N_1}(0)=u_{N_2}(0)
	\qquad \mathbb P\text{-a.s.}
	\]
Set $\tau:=\tau_{N_1}\wedge\tau_{N_2}.$
Then
\[
\mathbb{P}\Big(u_{N_1}(t)=u_{N_2}(t)\ \text{for all } t\in[0,\tau)\Big)=1.
\]
   \end{thm}
   \begin{proof}
      For each $M>0$, define
\[
\tau_M^i:=\inf\{t\ge0:\|u_{N_i}\|_{X_t}^2\ge M\}\wedge\tau_{N_i}, \qquad i=1,2.
\]
       Let $\tau_M=\tau_M^1\wedge\tau_M^2$ and $w=u_{N_1}-u_{N_2}$.   Let $0\leq \tau_a\le \tau_b\leq T\wedge{\tau_{M}}$
be arbitrary stopping times. 
By the definition of  local  solution, for every $t\in[\tau_a,\tau_b]$, we have, 
\begin{align*}
    w(t)=&w(\tau_a)+\int_{\tau_a}^t\Delta w(s)\,ds- \int_{\tau_a}^t\nabla\cdot\big(w(s)\nabla v_1(s)+{u_{N_2}(s)}\nabla\big(v_1(s)-v_2(s)\big)\big)\,ds\\&+\int_{\tau_a}^t\big(\sigma({u_{N_1}}(s))-\sigma({u_{N_2}(s)})\big)dW(s).
\end{align*}
Applying It\^o formula (\cite{liu2015stochastic}, Theorem 4.2.5) to $\|w(t)\|_{L^2(\mathcal {O})}^2$ for $t\in[\tau_a,\tau_b]$, we get 
\begin{align}\label{w_1}
    \nonumber&\|w(t)\|_{L^2(\mathcal {O})}^2-\|w(\tau_a)\|_{L^2(\mathcal {O})}^2\\
    \nonumber=&-2\int_{\tau_a}^t\big(\nabla w(s),\nabla w(s)\big)_{L^2(\mathcal {O})}\,ds\\\nonumber&+2 \int_{\tau_a}^t\big(\nabla w(s),w(s)\nabla v_1(s)+{u_{N_2}}(s)\nabla (v_1(s)-v_2(s))\big)_{L^2(\mathcal {O})}\,ds\\
    \nonumber&+2\int_{\tau_a}^t\big(w(s),\big(\sigma({u_{N_1}}(s))-\sigma({u_{N_2}(s)})\big)dW(s)\big)_{L^2(\mathcal {O})}\\\nonumber&+\int_{\tau_a}^t\|\sigma({u_{N_1}}(s))-\sigma({u_{N_2}(s)})\|_{L_2(U,L^2(\mathcal {O}))}^2\,ds\\
    \nonumber\leq&-2\int_{\tau_a}^t\|\nabla w(s)\|_{L^2(\mathcal {O})}^2\,ds+2 \int_{\tau_a}^t\|\nabla w(s)\|_{L^2(\mathcal {O})}\| w(s)\|_{L^2(\mathcal {O})}\|\nabla v_1(s)\|_{L^\infty(\mathcal {O})}\,ds\\&\nonumber+2 \int_{\tau_a}^t\|\nabla w(s)\|_{L^2(\mathcal {O})}\|{u_{N_2}}(s)\|_{L^4(\mathcal {O})}\|\nabla\big(v_1(s)-v_2(s)\big)\|_{L^{4}(\mathcal {O})}\,ds\\
\nonumber&+2\int_{\tau_a}^t\big(w(s),\big(\sigma({u_{N_1}}(s))-\sigma({u_{N_2}(s)})\big)dW(s)\big)_{L^2(\mathcal {O})}\\&+\int_{\tau_a}^t\|\sigma({u_{N_1}}(s))-\sigma({u_{N_2}(s)})\|_{L_2(U,L^2(\mathcal {O}))}^2\,ds.
\end{align}
 By Young's inequality,  there exists $\varepsilon_1>0$ such that
\begin{align}\label{w_2}
    \nonumber&2 \int_{\tau_a}^t\|\nabla w(s)\|_{L^2(\mathcal {O})}\| w(s)\|_{L^2(\mathcal {O})}\|\nabla v_1(s)\|_{L^\infty(\mathcal {O})}\,ds\\
    \le&C(\varepsilon_1)\int_{\tau_a}^t\|\nabla v_1(s)\|_{L^\infty(\mathcal {O})}^2\| w(s)\|_{L^2(\mathcal {O})}^2\,ds+\varepsilon_1\int_{\tau_a}^t\|\nabla w(s)\|_{L^2(\mathcal {O})}^2\,ds.
\end{align}
Since $t\leq\tau_b\leq\tau_M$, the definition of $\tau_M$
implies that  $\|{u_{N_1}}\|_{X_t}^2,\|{u_{N_2}}\|_{X_t}^2\le M$, then the Gagliardo-Nirenberg inequality yields \begin{align}\label{u_4}
&\nonumber\Big(\int_{\tau_a}^t \|u_{N_2}(s)\|_{L^4(\mathcal {O})}^4\,ds\Big)^{1/2}\\
\nonumber\le&
C\Big(\int_{0}^t \|u_{N_2}(s)\|_{L^2(\mathcal {O})}^2
\|\nabla u_{N_2}(s)\|_{L^2(\mathcal {O})}^2\,ds+\int_{0}^t \|u_{N_2}(s)\|_{L^2(\mathcal {O})}^4
\,ds\Big)^{1/2}\\\nonumber
\le&
C\Big(\sup_{s\in[{0},t]}\|u_{N_2}(s)\|_{L^2(\mathcal {O})}^2\Big)^{1/2}
\Big(\int_{0}^t \|\nabla u_{N_2}(s)\|_{L^2(\mathcal {O})}^2\,ds\Big)^{1/2}+C(T)\Big(\sup_{s\in[{0},t]}\|u_{N_2}(s)\|_{L^2(\mathcal {O})}^2\Big)\\
\le&
C(T)\Big(
\sup_{s\in[{0},t]}\|u_{N_2}(s)\|_{L^2(\mathcal {O})}^2
+
\int_{0}^t \|\nabla u_{N_2}(s)\|_{L^2(\mathcal {O})}^2\,ds
\Big)\le C(T,M).
\end{align}
Using Young's inequality, Lemma \ref{lem_v} and   \eqref{u_4}, we have that for any $\varepsilon_2>0$, 
\begin{align}\label{w_3}
  \nonumber& 2 \int_{\tau_a}^t\|\nabla w(s)\|_{L^2(\mathcal {O})}\|{u_{N_2}}(s)\|_{L^4(\mathcal {O})}\|\nabla\big(v_1(s)-v_2(s)\big)\|_{L^{4}(\mathcal {O})}\,ds\\
  \nonumber\le & 2\Big(\int_{\tau_a}^t\|\nabla w(s)\|_{L^2(\mathcal {O})}^2\,ds\Big)^\frac{1}{2}\Big(\int_{\tau_a}^t\| u_{N_2}(s)\|_{L^4(\mathcal {O})}^2\|\nabla \big(v_1(s)-v_2(s)\big)\|_{L^4(\mathcal {O})}^2\,ds\Big)^\frac{1}{2}\\\nonumber\le &C(\varepsilon_2)\Big(\int_{\tau_a}^t\| u_{N_2}(s)\|_{L^4(\mathcal {O})}^4ds\Big)^{\frac{1}{2}}\Big(\int_{\tau_a}^t\|\nabla\big(v_1(s)-v_2(s)\big)\|_{L^{4}(\mathcal {O})}^4ds\Big)^{\frac{1}{2}}+\varepsilon_2 \int_{\tau_a}^t\|\nabla w(s)\|_{L^2(\mathcal {O})}^2ds\\
  \nonumber
    \le & C(\varepsilon_2,M)\Big(\int_{\tau_a}^t\|w(s)\|_{L^2(\mathcal {O})}^4\,ds\Big)^\frac{1}{2}+C\|v_1(\tau_a)-v_2(\tau_a)\|_{H^1(\mathcal{O})}^2+\varepsilon_2\int_{\tau_a}^t\|\nabla w(s)\|_{L^2(\mathcal {O})}^2\,ds\\\nonumber 
    \le &\frac{1}{4}\sup_{s\in[{\tau_a},t]}\|w(s)\|_{L^2(\mathcal {O})}^2+C(\varepsilon_2,M)\int_{\tau_a}^t\| w(s)\|_{L^2(\mathcal {O})}^2\,ds+C\|v_1(\tau_a)-v_2(\tau_a)\|_{H^1(\mathcal{O})}^2\\&+\varepsilon_2\int_{\tau_a}^t\|\nabla w(s)\|_{L^2(\mathcal {O})}^2\,ds.
\end{align}
Hence, taking $\varepsilon_1+\varepsilon_2<\frac{1}{2}$, substituting (\ref{w_2}) and (\ref{w_3}) into (\ref{w_1}), we obtain
\begin{align}\label{w_4}
   \nonumber &\|w(t)\|_{L^2(\mathcal {O})}^2-\|w({\tau_a})\|_{L^2(\mathcal {O})}^2+\int_{\tau_a}^t\|\nabla w(s)\|_{L^2(\mathcal {O})}^2\,ds-C\|v_1(\tau_a)-v_2(\tau_a)\|_{H^1(\mathcal{O})}^2\\\nonumber
   \leq&\frac{1}{4}\sup_{s\in[{\tau_a},t]}\|w(s)\|_{L^2(\mathcal {O})}^2+C(M)\int_{\tau_a}^t\big(1+\|\nabla v_1(s)\|_{L^\infty(\mathcal {O})}^2\big)\| w(s)\|_{L^2(\mathcal {O})}^2\,ds\\\nonumber&+2\int_{\tau_a}^t\big(w(s),\big(\sigma({u_{N_1}(s)})-\sigma({u_{N_2}(s)})\big)dW(s)\big)_{L^2(\mathcal {O})}\\&+\int_{\tau_a}^t
\big\|\sigma({u_{N_1}(s)})-\sigma({u_{N_2}(s)}) \big\|_{L_2(U,L^2(\mathcal O))}^2\,ds.
\end{align}
For the parabolic equation \eqref{equ-v_t}, set $z:=v_1-v_2$.
Applying the $H^1(\mathcal{O})$-energy estimate to the equation satisfied by $z$, and then using the elliptic regularity estimate associated
with the boundary condition,
we obtain that for every \(t\in[\tau_a,\tau_b]\),
\begin{align}\label{est_V}
\|z(t)\|_{H^1(\mathcal O)}^2
+
c\int_{\tau_a}^{t}
\|z(s)\|_{H^2(\mathcal O)}^2\,ds
\leq
\|z(\tau_a)\|_{H^1(\mathcal O)}^2
+
C\int_{\tau_a}^{t}
\|w(s)\|_{L^2(\mathcal O)}^2\,ds,
\end{align}
where \(c>0\) and \(C>0\) are independent of
\(\tau_a,\tau_b\). Furthermore, combining \eqref{w_4} and \eqref{est_V},  it follows that
\begin{align}\label{w_6}
  \nonumber& \|w(t)\|_{L^2(\mathcal {O})}^2+\int_{\tau_a}^t\|\nabla w(s)\|_{L^2(\mathcal {O})}^2\,ds+\|z(t)\|_{H^1(\mathcal O)}^2\\ \nonumber
\leq&\|w({\tau_a})\|_{L^2(\mathcal {O})}^2+
C\|z(\tau_a)\|_{H^1(\mathcal O)}^2
+\frac{1}{4}\sup_{s\in[{\tau_a},t]}\|w(s)\|_{L^2(\mathcal {O})}^2\\ \nonumber&+C(M)\int_{\tau_a}^t\big(1+\|\nabla v_1(s)\|_{L^\infty(\mathcal {O})}^2\big)\| w(s)\|_{L^2(\mathcal {O})}^2\,ds\\ \nonumber&+2\int_{\tau_a}^t\big(w(s),\big(\sigma({u_{N_1}(s)})-\sigma({u_{N_2}(s)})\big)dW(s)\big)_{L^2(\mathcal {O})}\\&+\int_{\tau_a}^t
\big\|\sigma({u_{N_1}(s)})-\sigma({u_{N_2}(s)}) \big\|_{L_2(U,L^2(\mathcal O))}^2\,ds.
\end{align}
     Taking the supremum over $t\in[{\tau_a},{\tau_b}]$ and the expectation on both sides of \eqref{w_6}, then 
      using the
Burkholder-Davis-Gundy inequality, and Young's inequality to estimate the stochastic integral term, we get
\begin{align}\label{M_BDG}
    \nonumber &2\mathbb{E}\Big(\sup_{t\in[{\tau_a},{\tau_b}]}\Big|\int_{\tau_a}^{t}\big(w(s), \big(\sigma({u_{N_1}}(s))-\sigma({u_{N_2}(s)})\big)dW(s)\big)_{L^2(\mathcal {O})}\Big|\Big)\\\nonumber
\leq&C\mathbb{E}\Big(\int_{{\tau_a}}^{{\tau_b}}\|w(t)\|_{L^2(\mathcal {O})}^2\|\sigma({u_{N_1}}(t))-\sigma({u_{N_2}}(t))\|_{L_2(U,L^2(\mathcal {O}))}^2dt\Big)^\frac{1}{2}\\\nonumber
\leq&C\mathbb{E}\Big[\Big(\sup_{t\in[{\tau_a},{\tau_b}]}\|w(t)\|_{L^2(\mathcal {O})}^2\Big)^\frac{1}{2}\Big(\int_{{\tau_a}}^{{\tau_b}}\|\sigma ({u_{N_1}}(t))-\sigma ({u_{N_2}}(t))\|_{L_2(U,L^2(\mathcal {O}))}^2dt\Big)^\frac{1}{2}\Big]\\
    \leq&\frac{1}{4}\mathbb{E}\Big(\sup_{t\in[\tau_a,\tau_b]}\|w(t)\|_{L^2(\mathcal {O})}^2\Big)+C\mathbb{E}\Big(\int_{{\tau_a}}^{{\tau_b}}\|\sigma ({u_{N_1}}(t))-\sigma ({u_{N_2}}(t))\|_{L_2(U,L^2(\mathcal {O}))}^2dt\Big).
\end{align}
Applying Lemma \ref{lem_sigma}, we  find that for every $\varepsilon>0$,
\begin{align}\label{M_sigma}
    \nonumber&\int_{{\tau_a}}^{{\tau_b}}
\big\|\sigma({u_{N_1}(s)})-\sigma({u_{N_2}(s)}) \big\|_{L_2(U,L^2(\mathcal O))}^2\,ds\\\nonumber \le& \int_{{\tau_a}}^{{\tau_b}}C_\varepsilon({u_{N_1}},{u_{N_2}})\|{u_{N_1}(s)}-{u_{N_2}(s)}\|_{L^2(\mathcal O)}^2\,ds+\varepsilon\int_{{\tau_a}}^{{\tau_b}}\|\nabla{u_{N_1}(s)}-\nabla{u_{N_2}(s)}\|_{L^2(\mathcal O)}^2\,ds\\\le& C(M,\varepsilon)\int_{{\tau_a}}^{{\tau_b}}\|w(s)\|_{L^2(\mathcal O)}^2\,ds+\varepsilon\int_{{\tau_a}}^{{\tau_b}}\|\nabla w(s)\|_{L^2(\mathcal O)}^2\,ds.
\end{align}
 Combining \eqref{M_BDG} and  \eqref{M_sigma}, we obtain 
      \begin{align*}
&\mathbb{E}\Big(\sup_{t\in[\tau_a,\tau_b]}\|w({t})\|_{L^2(\mathcal {O})}^2\Big)+ 2\mathbb{E}\Big(\int_{{\tau_a}}^{{\tau_b}}\|\nabla w(t)\|_{L^2(\mathcal {O})}^2dt\Big)+C\mathbb{E}\Big(\sup_{t\in[\tau_a,\tau_b]}\|z({t})\|_{H^1(\mathcal {O})}^2\Big)\\\leq &\mathbb{E}\|w(\tau_a)\|_{L^2(\mathcal{O})}^2+C\mathbb{E}\|z(\tau_a)\|_{H^1(\mathcal{O})}^2+C\mathbb{E}\Big(\int_{{\tau_a}}^{{\tau_b}}\big(1+\|\nabla v_1(t)\|_{L^\infty (\mathcal{O})}^2\big)\|w(t)\|_{L^2(\mathcal {O})}^{2}dt\Big).
\end{align*}
Moreover, by Lemma \ref{lem_v} and the definition of $\tau_M$, we have 
\begin{align*}
    \int_{0}^{T\wedge\tau_M}\big(1+\|\nabla v_1(t)\|_{L^\infty (\mathcal{O})}^2\big)dt\le C(T,M,v_0),\qquad\mathbb{P}\text{-a.s.}
\end{align*}
Therefore, applying stochastic Gronwall's lemma \cite [Lemma 5.3]{glatt2009strong}, and since $u_{N_1}(0)=u_{N_2}(0)$ almost surely, $z(0)=0$,  we infer 
\[
\mathbb{E}\Big(\sup_{t\in[0,T\wedge\tau_M]}\|w({t})\|_{L^2(\mathcal {O})}^2\Big)+\mathbb{E}\Big(\int_0^{T\wedge{\tau_M}}\|\nabla w(t)\|_{L^2(\mathcal {O})}^2dt\Big)=0,
\]
which implies that 
\begin{align}\label{T_0}
    \mathbb{P}\Big(u_{N_1}(t)=u_{N_2}(t)\ \text{for all } t\in[0,T\wedge{\tau_M}]\Big)=1.
\end{align}

Since $T>0$ is arbitrary, it follows that
\[
\mathbb{P}\Big(
u_{N_1}(t)=u_{N_2}(t)\ \text{for all } t\in[0,\tau_M]
\Big)=1.
\]
Finally, let
\[
A_M:=\Big\{
u_{N_1}(t)=u_{N_2}(t)\ \text{for all } t\in[0,\tau_M]
\Big\}.
\]
Then $\mathbb{P}(A_M)=1$ for every $M\in\mathbb{N}$, and hence
\[
\mathbb{P}\Big(\bigcap_{M=1}^\infty A_M\Big)=1.
\]
Fix $\omega\in \bigcap_{M=1}^\infty A_M$ and let $t<\tau(\omega)$. Since $\tau_M(\omega)\uparrow\tau(\omega)$ as $M\to\infty$, there exists $M$ sufficiently large such that $t\le \tau_M(\omega)$. Therefore,
\[
u_{N_1}(t,\omega)=u_{N_2}(t,\omega).
\]
This proves that
\[
\mathbb{P}\Big(
u_{N_1}(t)=u_{N_2}(t)\ \text{for all } t\in[0,\tau)
\Big)=1.
\]
   \end{proof}
   We next establish the existence of a maximal strong solution to \eqref{equ1}. The argument follows ideas similar to those in \cite[Theorem~3.28]{kuehn2020pathwise} and \cite[Section~4]{chen2025well}. The result is formulated in the following theorem.  
\begin{thm}(Maximal existence)
    Under Assumption \ref{ass}, the system \eqref{equ1}  admits a unique maximal solution $(u,\tau^*)$ in the sense of Definition \ref{def_max}.
\end{thm}
\begin{proof}
For each $N\in\mathbb N$, let $(u_N,\tau_N)$ be the local strong solution to \eqref{equ1}, where
\[
\tau_N:=\inf\{t\ge0:\|u_N\|_{X_t}^2\ge N\}.
\] 
By the pathwise uniqueness of local strong solutions, for any $N_1<N_2$ we have
\[
u_{N_1}(t)=u_{N_2}(t)\qquad \text{for all } t\in[0,\tau_{N_1}\wedge{\tau_{N_2}}),\quad \mathbb P\text{-a.s.}
\]
We claim  that $\tau_{N_1}\le \tau_{N_2}$ almost surely. Indeed, by the definition of stopping times, on $\{\tau_{N_2}<\tau_{N_1}\}$,  we have
\[
N_2\le \|u_{N_2}\|_{X_{\tau_{N_2}}}^2=\|u_{N_1}\|_{X_{\tau_{N_2}}}^2<N_1,
\]
which contradicts $N_1<N_2$.
Hence $\{\tau_N\}_{N\in\mathbb N}$ is an increasing sequence of stopping times. Define
\[
\tau^*:=\lim_{N\to\infty}\tau_N.
\]
We now define a process $u$ on $[0,\tau^*)$ as follows. For 
	$t<\tau^*$, there exists $N\in\mathbb{N}$ such that $t<\tau_N$. We then  set
	\[
	u(t,\omega):=u_N(t,\omega).
	\]
	This definition is independent of the choice of $N$. Indeed, on 
	$t\in[0,\tau_{N_1}\wedge\tau_{N_2})$, pathwise uniqueness yields
	\[
	u_{N_1}(t,\omega)=u_{N_2}(t,\omega).
	\]
  	Thus $u$ is well defined on $[0,\tau^*)$. Moreover, for each fixed
	$N\in\mathbb N$,
	\[
	u(t)=u_N(t),
	\qquad 0\le t<\tau_N,
	\quad \mathbb P\text{-a.s.}
	\]
	Since $(u_N,\tau_N)$ is a local strong solution to \eqref{equ1}, it follows
	that $(u,\tau_N)$ is also a local strong solution to \eqref{equ1}.

	We next verify the blow-up criterion. For  $N\in\mathbb N$, by the definition
	of $\tau_N$, on the set $\{\tau_N<\infty\}$ one has
    \[
\limsup_{t\uparrow\tau_N}
\left[
\sup_{0\le s\le t}\|u_N(s)\|_{L^2(\mathcal O)}^2
+
\int_0^t\|\nabla u_N(s)\|_{L^2(\mathcal O)}^2\,ds
\right]\ge N,
\qquad\mathbb P\text{-a.s.}
\]
    Since $u=u_N$ on $[0,\tau_N)$, it follows that
	\[
\limsup_{t\uparrow\tau_N}
\left[
\sup_{0\le s\le t}\|u(s)\|_{L^2(\mathcal O)}^2
+
\int_0^t\|\nabla u(s)\|_{L^2(\mathcal O)}^2\,ds
\right]
\ge N
\quad\text{on }\{\tau_N<\infty\},
\qquad\mathbb P\text{-a.s.}
\]
By $\tau_N\le \tau^*$ and  $\{{\tau^*}<\infty\}\subset \{\tau_N<\infty\}$, and since \(N\) can be chosen arbitrarily
large, we conclude that
\[
\limsup_{t\uparrow\tau^*}
\left[
\sup_{0\le s\le t}\|u(s)\|_{L^2(\mathcal O)}^2
+
\int_0^t\|\nabla u(s)\|_{L^2(\mathcal O)}^2\,ds
\right]
=\infty
\quad\text{on }\{\tau^*<\infty\},
\qquad\mathbb P\text{-a.s.}
\]
    	Therefore, \((u,\tau^*)\) is a maximal local strong solution to
\eqref{equ1} in the sense of Definition~\ref{def_max}.
\end{proof}
\subsection{Uniform estimates}  We  now  established the estimates  for local strong solutions that are uniform in $N$.
  \begin{thm}\label{thm_estimate}
      Let $(u,\tau_N)$ be a local strong solution  of  system (\ref{equ1}). Under Assumption \ref{ass},  for every $T>0$ and every $\alpha\in\big(\max\{0,\frac{2-r}{2}\},\frac{1}{2}\big)$, there exists a constant $C(T)>0$, independent of $N$, such that 
      \begin{align}\label{est_C(T)}
	    \mathbb E\Big(\sup_{0\le t\le T}(1+\|u(t\wedge\tau_N)\|_{L^2(\mathcal{O})}^2)^\alpha\Big)+\mathbb{E}\Big(\int_0^{T\wedge{\tau_N}}
		\frac{\|\nabla u(s)\|_{L^2(\mathcal{O})}^2}{(1+\|u(s)\|_{L^2(\mathcal{O})}^2)^{1-\alpha}}\,ds\Big)
	\le C(T).
	\end{align}
  \end{thm} 

  \begin{proof}
    	 Applying It\^o's formula to $\|u(t\wedge{\tau_N})\|_{L^2(\mathcal O)}^2$, for convenience, we denote the $L^2(\mathcal{O})$ norm of $u$ and $\nabla u$  simply by $\|u\|_2$ and $\|\nabla u\|_2$ respectively and we obtain
        \begin{align}\label{es_L2}
		\nonumber&\|u(t\wedge{\tau_N})\|_2^2-\|u_0\|_2^2\\
		=&2\int_0^{t\wedge{\tau_N}}\big(u(s),\Delta u(s)\big)_2\,ds
		-2\int_0^{t\wedge{\tau_N}}\big(u(s),\nabla\cdot(u(s)\nabla v(s))\big)_2\,ds
		\nonumber\\
&+B_0^2\int_0^{t\wedge{\tau_N}}\|u(s)\|_p^{2r}\|u(s)\|_2^2\,ds+2\sum_{k=1}^\infty\int_0^{t\wedge{\tau_N}}b_k\|u(s)\|_p^r\|u(s)\|_2^2dW_k(s).
	\end{align}
  We introduce the Lyapunov function $\Phi(x)=(1+x)^\alpha$ with $\alpha\in\big(\max\{0,\frac{2-r}{2}\},\frac{1}{2}\big).$ Applying It\^o's formula to $\Phi(\|u(t\wedge{\tau_N})\|_2^2)$ and combining \eqref{es_L2} yields
	\begin{align*}
		\nonumber&\big(1+\|u(t\wedge{\tau_N})\|_2^2\big)^\alpha
+2\alpha\int_0^{t\wedge{\tau_N}}\frac{\|\nabla u(s)\|_2^2}{(1+\|u(s)\|_2^2)^{1-\alpha}}\,ds
		\\\nonumber
		=&(1+\|u_0\|_2^2)^\alpha
		+2\alpha\int_0^{t\wedge{\tau_N}}
		\frac{\int_{\mathcal O}u(s)\nabla u(s)\cdot \nabla v(s)dx}
		{(1+\|u(s)\|_2^2)^{1-\alpha}}\,ds
+\alpha\int_0^{t\wedge{\tau_N}}
		\frac{B_0^2\|u(s)\|_p^{2r}\|u(s)\|_2^2}
		{(1+\|u(s)\|_2^2)^{1-\alpha}}\,ds
		\\
		&-2\alpha(1-\alpha)\int_0^{t\wedge{\tau_N}}
		\frac{{B_0^2}\|u(s)\|_p^{2r}\|u(s)\|_2^4}
		{(1+\|u(s)\|_2^2)^{2-\alpha}}\,ds
+2\alpha\sum_{k=1}^\infty\int_0^{t\wedge{\tau_N}}
		\frac{b_k\|u(s)\|_p^r\|u(s)\|_2^2}
		{(1+\|u(s)\|_2^2)^{1-\alpha}}dW_k(s).
	\end{align*}
    Taking $2<m\le \frac{(r+2)p}{p+r}$ and  by H\"older's and Young's inequalities,  we get  for any $\varepsilon_1>0$,
	\begin{align}\label{ine_v}
	    2\alpha\int_{\mathcal O}u\nabla u\cdot \nabla v\,dx
	\le
	\varepsilon_1\|\nabla u\|_2^2
	+C\|u\|_m^2\|\nabla v\|_{\frac{2m}{m-2}}^2.
	\end{align}
	Substituting \eqref{ine_v} into the above equality, we obtain
	\begin{align}
		\label{L3}
		\nonumber&(1+\|u(t\wedge{\tau_N})\|_2^2)^\alpha
		+(2\alpha-\varepsilon_1)\int_0^{t\wedge{\tau_N}}\frac{\|\nabla u(s)\|_2^2}{(1+\|u(s)\|_2^2)^{1-\alpha}}\,ds\\\nonumber&+(\alpha-2\alpha^2)\int_0^{t\wedge{\tau_N}}
		\frac{{B_0^2}\|u(s)\|_p^{2r}\|u(s)\|_2^4}
		{(1+\|u(s)\|_2^2)^{2-\alpha}}\,ds
		\nonumber\\
		\le&(1+\|u_0\|_2^2)^\alpha
		+C\int_0^{t\wedge{\tau_N}}
		\frac{\|u(s)\|_m^2\|\nabla v(s)\|_{\frac{2m}{m-2}}^2}
		{(1+\|u(s)\|_2^2)^{1-\alpha}}\,ds
		+\alpha\int_0^{t\wedge{\tau_N}}
		\frac{{B_0^2}\|u(s)\|_p^{2r}\|u(s)\|_2^2}
		{(1+\|u(s)\|_2^2)^{2-\alpha}}\,ds
		\nonumber\\
		&
+2\alpha\sum_{k=1}^\infty\int_0^{t\wedge{\tau_N}}
		\frac{b_k\|u(s)\|_p^r\|u(s)\|_2^2}
		{(1+\|u(s)\|_2^2)^{1-\alpha}}dW_k(s).
	\end{align}
    Next, we use the Gagliardo-Nirenberg inequality, the interpolation inequality, and Young's inequality to estimate the  integral terms on the right-hand side in \eqref{L3} so that they can be absorbed into the left-hand side in \eqref{L3}. 
    
    We claim that for any $\varepsilon_2,\varepsilon_3>0$,
	\begin{equation}
		\label{es_p,2}
		\alpha{B_0^2}\|u\|_p^{2r}\|u\|_2^2
		\le
		\varepsilon_2\|\nabla u\|_2^2\|u\|_2^2
		+\varepsilon_3 \alpha{B_0^2}\|u\|_p^{2r}\|u\|_2^4
		+C(\varepsilon_2,\varepsilon_3,\alpha,{B_0^2}).
	\end{equation}
	Indeed,   if $\|u\|_2^2\ge \varepsilon_3^{-1}$, then
	\[
	\alpha{B_0^2}\|u\|_p^{2r}\|u\|_2^2
	\le
	\varepsilon_3 \alpha{B_0^2}\|u\|_p^{2r}\|u\|_2^4.
	\]
    If $\|u\|_2^2<\varepsilon_3^{-1}$, by the Gagliardo--Nirenberg inequality in two dimensions,
	\begin{align*}
		\|u\|_p^{2r}\|u\|_2^2
		&\le
		C\|\nabla u\|_2^{\frac{2r(p-2)}{p}}\|u\|_2^{\frac{2p+4r}{p}}
		+C\|u\|_2^{2r+2}\\
		&=
		C\|u\|_2^{\frac{2(p+4r-rp)}{p}}
		\big(\|\nabla u\|_2^2\|u\|_2^2\big)^{\frac{r(p-2)}{p}}
		+C\|u\|_2^{2r+2}.
	\end{align*}
	Since $p>2$ and $r<\frac{p}{p-2}$, we have
	\[
	\frac{r(p-2)}{p}<1,
	\qquad
	\frac{2(p+4r-rp)}{p}>0.
	\]
	 Therefore,
	\[
	\|u\|_p^{2r}\|u\|_2^2
	\le
	C(\varepsilon_3)\big(\|\nabla u\|_2^2\|u\|_2^2\big)^{\frac{r(p-2)}{p}}
	+C(\varepsilon_3),
	\]
	and Young's inequality yields
	\[
	\alpha{B_0^2}\|u\|_p^{2r}\|u\|_2^2
	\le
	\varepsilon_2\|\nabla u\|_2^2\|u\|_2^2
	+C(\varepsilon_2,\varepsilon_3,\alpha,{B_0^2}).
	\]
	Combining the above two cases, we obtain \eqref{es_p,2}.
	Furthermore, using
	\[
	\frac{\|u\|_2^2}{(1+\|u\|_2^2)^{2-\alpha}}
	\le
	\frac{1}{(1+\|u\|_2^2)^{1-\alpha}},
	\qquad
	\frac{1}{(1+\|u\|_2^2)^{2-\alpha}}\le 1,
	\]
	we infer that
	\begin{align}
		\label{es_a,q,2}
		\alpha\frac{{B_0^2}\|u\|_p^{2r}\|u\|_2^2}{(1+\|u\|_2^2)^{2-\alpha}}
		\le
		\varepsilon_2\frac{\|\nabla u\|_2^2}{(1+\|u\|_2^2)^{1-\alpha}}
		+\varepsilon_3 \alpha\frac{{B_0^2}\|u\|_p^{2r}\|u\|_2^4}{(1+\|u\|_2^2)^{2-\alpha}}
+C(\varepsilon_2,\varepsilon_3).
	\end{align}
	We next estimate the term involving $\|u\|_m^2\|\nabla v\|_{\frac{2m}{m-2}}^2$.
	By Young's inequality, we have
	\begin{align}
		\label{u-v-term}
		C\frac{\|u\|_m^2\|\nabla v\|_{\frac{2m}{m-2}}^2}{(1+\|u\|_2^2)^{1-\alpha}}
		\le
		C\frac{\|u\|_m^{\frac{2(2-\alpha)}{1-\alpha}}}{(1+\|u\|_2^2)^{2-\alpha}}
		+\|\nabla v\|_{\frac{2m}{m-2}}^{2(2-\alpha)}.
	\end{align}
	Since $2<m\le \frac{(r+2)p}{p+r},$
	taking $\theta=\frac{r}{r+2}$ yields $\frac1m\ge \frac{p+r}{(r+2)p}= \frac{\theta}{p}+\frac{1-\theta}{2}.$
	Hence, by interpolation inequality on bounded domains, it holds 
	\begin{align}\label{interpolation_1}
	    \|u\|_m^{\frac{2(2-\alpha)}{1-\alpha}}\le C\|u\|_{\frac{(r+2)p}{p+r}}^{\frac{2(2-\alpha)}{1-\alpha}}
	\le
	C\|u\|_p^{\frac{2(2-\alpha)r}{(1-\alpha)(r+2)}}
	\|u\|_2^{\frac{4(2-\alpha)}{(1-\alpha)(r+2)}}
	=
	C\big(\|u\|_p^{2r}\|u\|_2^4\big)^{\frac{2-\alpha}{(1-\alpha)(r+2)}}.
	\end{align}
	Since $\alpha<\frac{1}{2}<\frac{r}{r+1}$, we have $\frac{2-\alpha}{(1-\alpha)(r+2)}<1.$
	Therefore, using Young's inequality arrives
	\begin{align}
		\label{u_m_a}
		C\frac{\|u\|_m^{\frac{2(2-\alpha)}{1-\alpha}}}{(1+\|u\|_2^2)^{2-\alpha}}
		\le
		\varepsilon_4 \alpha\frac{{B_0^2}\|u\|_p^{2r}\|u\|_2^4}{(1+\|u\|_2^2)^{2-\alpha}}
		+C(\varepsilon_4).
	\end{align}
	For $\int_0^{t\wedge{\tau_N}}\|\nabla v(s)\|_{\frac{2m}{m-2}}^{2(2-\alpha)}\,ds$,  by Lemma \ref{lem_v}, we know 
	\begin{equation}
		\int_0^{t\wedge{\tau_N}}\|\nabla v(s)\|_{\frac{2m}{m-2}}^{2(2-\alpha)}\,ds
		\le
		Ct\|\nabla v_0\|_{\frac{2m}{m-2}}^{2(2-\alpha)}+C\int_0^{t\wedge{\tau_N}}\|u(s)\|_m^{2(2-\alpha)}\,ds.
	\end{equation}
	We further claim that for any $\varepsilon_5>0$,
	\begin{equation}
		\label{ine_u,m}
C\|u\|_m^{2(2-\alpha)}
\le\varepsilon_5
		\frac{\alpha{B_0^2}\|u\|_p^{2r}\|u\|_2^4}{(1+\|u\|_2^2)^{2-\alpha}}
		+
		C(\varepsilon_5).
	\end{equation}
    Indeed, if  $\|u\|_2^2\le 1$,
	then
	\[
	(1+\|u\|_2^2)^{2-\alpha}\le 2^{2-\alpha},
	\]
	and hence
	\[
	\frac{\alpha{B_0^2}\|u\|_p^{2r}\|u\|_2^4}{(1+\|u\|_2^2)^{2-\alpha}}
	\ge
	2^{-(2-\alpha)}\alpha{B_0^2}\|u\|_p^{2r}\|u\|_2^4.
	\]
	Similarly with \eqref{interpolation_1}, we use again the interpolation inequality with $\theta=\frac{r}{r+2}$ to get
	\[
	\|u\|_m^{2(2-\alpha)}
	\le C\|u\|_{\frac{(r+2)p}{p+r}}^{2(2-\alpha)}\le
	C\|u\|_p^{\frac{2(2-\alpha)r}{r+2}}
	\|u\|_2^{\frac{4(2-\alpha)}{r+2}}
	=
	C(\|u\|_p^{2r}\|u\|_2^4)^{\frac{2-\alpha}{r+2}}.
	\]
	Since $\frac{2-\alpha}{r+2}<1$, Young's inequality implies
	\begin{align}
	    \|u\|_m^{2(2-\alpha)}
	\le
	\varepsilon_5\frac{\alpha{B_0^2}\|u\|_p^{2r}\|u\|_2^4}{(1+\|u\|_2^2)^{2-\alpha}}
	+
	C(\varepsilon_5).
	\end{align}
	 If $\|u\|_2^2>1$, then
	\[
	(1+\|u\|_2^2)^{2-\alpha}\le 2^{2-\alpha}\|u\|_2^{2(2-\alpha)},
	\]
	and therefore
	\[
	\frac{\alpha{B_0^2}\|u\|_p^{2r}\|u\|_2^4}{(1+\|u\|_2^2)^{2-\alpha}}
	\ge
	2^{-(2-\alpha)}\alpha{B_0^2}\|u\|_p^{2r}\|u\|_2^{2\alpha}.
	\]
	Since
	\[
	m\le \frac{(r+2)p}{p+r}<\frac{2p(r+\alpha)}{2r+\alpha p},
	\]
	we may take $\theta=\frac{r}{r+\alpha}$ so that
	\[
	\frac1m\ge \frac{\theta}{p}+\frac{1-\theta}{2},
	\]
	Consequently, using the interpolation inequality yields
	\[
	\|u\|_m^{2(2-\alpha)}
	\le
	C\|u\|_p^{\frac{2(2-\alpha)r}{r+\alpha}}
	\|u\|_2^{\frac{2\alpha(2-\alpha)}{r+\alpha}}
	=
	C(\|u\|_p^{2r}\|u\|_2^{2\alpha})^{\frac{2-\alpha}{r+\alpha}}.
	\]
	Since $\frac{2-r}{2}<\alpha$, we have $\frac{2-\alpha}{r+\alpha}<1.$
	Thus, Young's inequality  gives
	\begin{align}
	    \|u\|_m^{2(2-\alpha)}
	\le
	\varepsilon_5\frac{\alpha{B_0^2}\|u\|_p^{2r}\|u\|_2^4}{(1+\|u\|_2^2)^{2-\alpha}}
	+
	C(\varepsilon_5).
	\end{align}
	Combining the above two cases, \eqref{ine_u,m} holds.
Thus, it follows  that
\begin{equation}
		\label{ine_v,m}
		\int_0^{t\wedge{\tau_N}}\|\nabla v(s)\|_{\frac{2m}{m-2}}^{2(2-\alpha)}\,ds
		\le
\varepsilon_5\int_0^{t\wedge{\tau_N}}\frac{\alpha{B_0^2}\|u(s)\|_p^{2r}\|u(s)\|_2^4}{(1+\|u(s)\|_2^2)^{2-\alpha}}\,ds+C(v_0,\varepsilon_5,t).
	\end{equation}
	Substituting  \eqref{es_a,q,2}, \eqref{u-v-term}, \eqref{u_m_a} and  \eqref{ine_v,m}
	into \eqref{L3}, we arrive at
	\begin{align}
		\label{La}
		\nonumber&\big(1+\|u(t\wedge{\tau_N})\|_2^2\big)^\alpha
		+(2\alpha-\varepsilon_1-\varepsilon_2)\int_0^{t\wedge{\tau_N}}
		\frac{\|\nabla u(s)\|_2^2}{(1+\|u(s)\|_2^2)^{1-\alpha}}\,ds\\\nonumber&+\bigl(\alpha-2\alpha^2-\varepsilon_3-\varepsilon_4-\varepsilon_5\bigr)
		\int_0^{t\wedge{\tau_N}}
		\frac{{B_0^2}\|u(s)\|_p^{2r}\|u(s)\|_2^4}
		{(1+\|u(s)\|_2^2)^{2-\alpha}}\,ds
		\nonumber\\
		\le&\big(1+\|u_0\|_2^2\big)^\alpha+2\alpha\sum_{k=1}^\infty\int_0^{t\wedge{\tau_N}}
		\frac{b_k\|u(s)\|_p^r\|u(s)\|_2^2}{(1+\|u(s)\|_2^2)^{1-\alpha}}dW_k(s)
		+C(\varepsilon_1, \varepsilon_2,\varepsilon_3,\varepsilon_4,\varepsilon_5,t,v_0).
	\end{align}
    Choosing $\varepsilon_i>0$, $1\le i\le 5$ sufficiently small such that
	\[
	2\alpha-\varepsilon_1-\varepsilon_2>0,
	\qquad
	\alpha-2\alpha^2-\varepsilon_3-\varepsilon_4-\varepsilon_5>0.
	\]
	Taking expectations in \eqref{La}, and using the definition of $\tau_N$, we see that the stochastic integral is a square-integrable  martingale with zero expectation. Then, we obtain
	\begin{align}
		\label{exp_u}
		\mathbb E\Big(\int_0^{t\wedge{\tau_N}}
		\frac{{B_0^2}\|u(s)\|_p^{2r}\|u(s)\|_2^4}
		{(1+\|u(s)\|_2^2)^{2-\alpha}}\,ds\Big)
		\le C(t,v_0,u_0).
	\end{align}
     Furthermore, taking the supremum over $t\in[0,T]$ in \eqref{La}, then taking expectations, we get
	\begin{align}
		\label{est_sup_u}
		&\mathbb E\Big(\sup_{0\le t\le T}(1+\|u(t\wedge{\tau_N})\|_2^2)^\alpha\Big)+\mathbb{E}\Big(\int_0^{T\wedge{\tau_N}}
		\frac{\|\nabla u(s)\|_2^2}{(1+\|u(s)\|_2^2)^{1-\alpha}}\,ds\Big)
		\nonumber\\
		&\le
		\mathbb E\big(\big(1+\|u_0\|_2^2\big)^\alpha\big)
		+2\alpha\mathbb E\Big(\sup_{0\le t\le T}
		\Big|
		\sum_{k=1}^\infty\int_0^{t\wedge{\tau_N}}
		\frac{b_k\|u(s)\|_p^r\|u(s)\|_2^2}{(1+\|u(s)\|_2^2)^{1-\alpha}}dW_k(s)
		\Big|\Big)
		\nonumber\\
&+C(T,v_0,u_0).
	\end{align}
	By the Burkholder-Davis-Gundy inequality and Young's inequality,  we have 
	\begin{align}
		\label{BDG-estimate}
		&2\alpha\mathbb E\Big(\sup_{0\le t\le T}
		\Big|
		\sum_{k=1}^\infty\int_0^{t\wedge{\tau_N}}
		\frac{b_k\|u(s)\|_p^r\|u(s)\|_2^2}{(1+\|u(s)\|_2^2)^{1-\alpha}}dW_k(s)
		\Big|\Big)
		\nonumber\\
		\le&
		C\mathbb E\Big[\Big(
		\int_0^{T\wedge{\tau_N}}
		\frac{{B_0^2}\|u(s)\|_p^{2r}\|u(s)\|_2^4}
		{(1+\|u(s)\|_2^2)^{2-2\alpha}}\,ds
		\Big)^{1/2}\Big]
		\nonumber\\
		=&
		C\mathbb E\Big[\Big(
		\int_0^{T\wedge{\tau_N}}
		\big(1+\|u(s)\|_2^2\big)^\alpha
		\frac{{B_0^2}\|u(s)\|_p^{2r}\|u(s)\|_2^4}
		{(1+\|u(s)\|_2^2)^{2-\alpha}}\,ds
		\Big)^{1/2}\Big]
		\nonumber\\
		\le&
		C\mathbb E\Big[
		\Big(
		\sup_{0\le t\le T}(1+\|u(t\wedge{\tau_N})\|_2^2)^\alpha
		\Big)^{1/2}
		\Big(
		\int_0^{T\wedge{\tau_N}}
		\frac{{B_0^2}\|u(s)\|_p^{2r}\|u(s)\|_2^4}
		{(1+\|u(s)\|_2^2)^{2-\alpha}}\,ds
		\Big)^{1/2}
		\Big]
		\nonumber\\
		\le&
		\frac12\mathbb E\Big(\sup_{0\le t\le T}(1+\|u(t\wedge{\tau_N})\|_2^2)^\alpha\Big)
		+C\mathbb E\Big(\int_0^{T\wedge{\tau_N}}
		\frac{{B_0^2}\|u(s)\|_p^{2r}\|u(s)\|_2^4}
		{(1+\|u(s)\|_2^2)^{2-\alpha}}\,ds\Big)
		\nonumber\\
		\le&
		\frac12\mathbb E\Big(\sup_{0\le t\le T}(1+\|u(t\wedge{\tau_N})\|_2^2)^\alpha\Big)
		+C(T,v_0,u_0),
	\end{align}
	where we have used the estimate \eqref{exp_u} in the last step.  Substituting \eqref{BDG-estimate} into \eqref{est_sup_u}, we conclude that
	\begin{align}
	    \mathbb E\Big(\sup_{0\le t\le T}(1+\|u(t\wedge\tau_N)\|_2^2)^\alpha\Big)+\mathbb{E}\Big(\int_0^{T\wedge{\tau_N}}
		\frac{\|\nabla u(s)\|_2^2}{(1+\|u(s)\|_2^2)^{1-\alpha}}\,ds\Big)
	\le C(T,v_0,u_0).
	\end{align}
    This completes the proof of Theorem \ref{thm_estimate}.
     \end{proof}
     We are now in a position to complete the proof of the main theorem.
    \begin{proof}[The proof of Theorem \ref{thm1}] 
Combining Theorem \ref{thm_Local} and Theorem\ref{thm3.2}, we  obtain a maximal local strong solution $(u,\tau^*)$. We use the uniform estimate \eqref{est_C(T)} in Theorem \ref{thm_estimate} to conclude that the solution is global and pathwise unique. Applying  Markov's inequality yields
    \begin{align*}
       & \mathbb{P}\Big\{\sup_{t\in[0,T]}\|u({t\wedge{\tau_N}})\|_{L^2(\mathcal O)}^2\ge \frac{N}{2}\Big\}\\\le& \mathbb{P}\Big\{\Big(1+\sup_{t\in[0,T]}\|u({t\wedge{\tau_N}})\|_{L^2(\mathcal O)}^2\Big)^\alpha\ge \Big(\frac{N}{2}\Big)^\alpha\Big\}\le \frac{2^\alpha}{N^\alpha}\mathbb{E}\Big(\sup_{0\le t\le T}(1+\|u(t\wedge\tau_N)\|_2^2)^\alpha\Big)\\\le& \frac{C(T)}{N^\alpha}.
    \end{align*}
    Similarly, we deduce
    \begin{align*}
    &\mathbb{P}\Big\{\int_0^{T\wedge{\tau_N}}
		\|\nabla u(s)\|_2^2\,ds\ge \frac{N}{2}\Big\}\\\le& 
        \mathbb{P}\Big\{\Big(1+\sup_{t\in[0,T]}\|u({t\wedge{\tau_N}})\|_{L^2(\mathcal O)}^2\Big)^{1-\alpha}\int_0^{T\wedge{\tau_N}}
		\frac{\|\nabla u(s)\|_2^2}{(1+\|u(s)\|_2^2)^{1-\alpha}}\,ds\ge \frac{N}{2}\Big\}\\
        \le&\mathbb{P}\Big\{\Big(1+\sup_{t\in[0,T]}\|u({t\wedge{\tau_N}})\|_{L^2(\mathcal O)}^2\Big)^{\alpha}\ge\Big(\frac{N}{2}\Big)^\alpha\Big\}+\mathbb{P}\Big\{\int_0^{T\wedge{\tau_N}}
		\frac{\|\nabla u(s)\|_2^2}{(1+\|u(s)\|_2^2)^{1-\alpha}}ds\ge \Big(\frac{N}{2}\Big)^\alpha\Big\}\\\le&\frac{2^\alpha}{N^\alpha}\mathbb{E}\Big(\sup_{0\le t\le T}(1+\|u(t\wedge\tau_N)\|_2^2)^\alpha\Big)+\frac{2^\alpha}{N^\alpha}\mathbb{E}\Big(\int_0^{T\wedge{\tau_N}}
		\frac{\|\nabla u(s)\|_2^2}{(1+\|u(s)\|_2^2)^{1-\alpha}}\,ds\Big)\\
        \le&\frac{C(T)}{N^\alpha}.
    \end{align*}
Furthermore, by the definition of stopping time ${\tau_N}$, it follows that
       \begin{align*}
       \nonumber\mathbb{P}\Big\{{\tau_N}\leq T\Big\}&=\mathbb{P}\Big\{\sup_{t\in[0,T]}\|u({t\wedge{\tau_N}})\|_{L^2(\mathcal O)}^2+\int_0^{T\wedge{\tau_N}}\|\nabla u(s)\|_{L^2(\mathcal{O})}^2\,ds\ge N\Big\}\\&\le \mathbb{P}\Big\{\sup_{t\in[0,T]}\|u({t\wedge{\tau_N}})\|_{L^2(\mathcal O)}^2\ge \frac{N}{2}\Big\}+\mathbb{P}\Big\{\int_0^{T\wedge{\tau_N}}\|\nabla u(s)\|_{L^2(\mathcal{O})}^2\,ds\ge\frac {N}{2}\Big\}
\\&\leq \frac{C(T)}{N^\alpha}\to 0,~(N\to \infty).
       \end{align*}
       
       Since $\tau_N\uparrow\tau^*$, we have
	\[
	\{\tau^*\le T\}=\bigcap_{N\in\mathbb N}\{\tau_N\le T\}.
	\]
	Therefore, by continuity of probability,
	\[
	\mathbb P(\tau^*\le T)=\lim_{N\to\infty}\mathbb P(\tau_N\le T)=0.
	\]
	As $T>0$ is arbitrary, it follows that $\mathbb P(\tau^*<\infty)=0$, i.e. $\tau^*=\infty$ $\mathbb P$-a.s. Therefore, the maximal local strong solution is global. The pathwise uniqueness follows from Theorem \ref{thm3.2}. Hence system \eqref{equ1} admits a global pathwise unique strong solution.
\end{proof}

    \nocite{*}
	\bibliography{refer} 
	\bibliographystyle{plain}  

\end{document}